\documentclass[11pt,a4paper]{amsart}

\usepackage{amsmath,mathbbol,amssymb,mathbbol}
\usepackage[T1]{fontenc} \usepackage{mathpazo}
\usepackage[active]{srcltx}

\usepackage{upref}

\usepackage{amsthm}

\newtheorem{theorem}{Theorem}[section]
\newtheorem{proposition}[theorem]{Proposition}

\newtheorem{lemma}[theorem]{Lemma}
\newtheorem{corollary}[theorem]{Corollary}
\theoremstyle{definition}
\newtheorem{definition}[theorem]{Definition}

\newtheorem{remark}[theorem]{Remark}

\def\bX{\mathbb{X}}

\def\R{\mathbb{R}}

\def\N{\mathbb{N}}

\def\C{\mathbb{C}}
\def\ZZ{\mathbb{Z}}

\def\E{\mathbb{E}}

\def\lla{\lambda}

\def \supp{\operatorname{supp}}

\def\cA{{\mathcal A}}

\def\cF{{\mathcal F}}

\def\cL{{\mathcal L}}
\def\cM{{\mathcal M}}

\def\cR{{\mathcal R}}

\newcommand{\X}{{X}}
\newcommand{\Y}{{Y}}

\def\trXY1-tq{Tr(X,Y, 1-\theta ,q)}

\newcommand{\ud}{\mathrm{d}}

\title[Weighted inequalities for singular integral operators]{Weighted inequalities for singular integral operators on the half-line}

\author{Ralph Chill}

\address{R.~Chill, Institut f\"ur Analysis, Fachrichtung Mathematik, TU Dresden, 01062 Dresden, Germany}
\email{ralph.chill@tu-dresden.de}

\author{Sebastian Kr\'ol}
\address{S. Kr\'ol, Institut f\"ur Analysis, Fachrichtung Mathematik, TU Dresden, 01062 Dresden, Germany, and Faculty of Mathematics and Computer Science, Nicolaus Copernicus University, ul. Chopina 12/18, 87-100 Toru\'n, Poland}
\email{sebastian.krol@mat.umk.pl}

\thanks{The second author is grateful for support by the Alexander von Humboldt Foundation}

\numberwithin{equation}{section}

\begin{document}

\date{\today}



\begin{abstract}
We prove weighted estimates for singular integral operators which operate on function spaces on a half-line. The class of admissible weights includes Muckenhoupt weights and weights satisfying Sawyer's one-sided conditions. The kernels of the operators satisfy relaxed Dini conditions. We apply the weighted estimates to extrapolation of maximal $L^p$ regularity of first order, second order and fractional order Cauchy problems into weighted rearrangement invariant Banach function spaces. In particular, we provide extensions, as well as a unification of recent results due to Auscher and Axelsson, and Chill and Fiorenza.
\end{abstract}

\renewcommand{\subjclassname}{\textup{2010} Mathematics Subject Classification}

\maketitle

\section{Motivation and description of the main result} 

Weighted $L^p$-estimates for singular integral operators play an important role in harmonic analysis and in its applications to elliptic and parabolic partial differential equations, as well as abstract evolution equations on Banach spaces, where singular integral operators with operator valued kernels naturally arise via representation formulas for solutions. These applications are in turn important for the theory of nonlinear equations where one is often interested in obtaining wellposedness and regularity results for data in various function spaces. \\

A prominent part of the literature is devoted to weighted inequalities for the Hardy-Littlewood maximal function and singular integral operators on $\R^n$. Motivated in particular by applications to $L^p$-maximal regularity, or actually $L^p_w$-maximal regularity and $\E_w$-maximal regularity ($\E$ being a rearrangement invariant Banach function space), of first order, second order and fractional order Cauchy problems on the half-line, we study in this article weighted inequalities for singular integral operators on the half-line. This special situation brings in some new features which we combine with techniques which have been developed recently in the context of singular integral operators on function spaces on the line or on $\R^n$. By concentrating on the half-line case, the natural context in the above mentioned applications, we obtain weighted inequalities for a larger class of weights, which includes the class of Muckenhoupt weights and the class of weights satisfying the one-sided Sawyer condition on the line. Second, we obtain weighted inequalities for singular integral operators with kernels which satisfy comparatively weak regularity conditions. Third, by relying on Rubio de Francia's extrapolation technique, we obtain not only weighted $L^p$-estimates, but an extrapolation result in the context of rearrangement invariant Banach function spaces. \\

Let us explain these results in more detail and prepare the notational background, by recalling first the situation of singular integral operators on the line. Throughout, let $\X$ and $\Y$ be two Banach spaces with norms $|\cdot |_\X$ and $|\cdot |_Y$, respectively. A measurable function $K:\R\times \R \rightarrow \cL(\X ,\Y )$ is called a {\em kernel} if $K(t, \cdot) \in L^1_{loc}(\R \setminus \{t\}; \cL(\X ,\Y ))$ for every $t \in \R$. 
We say that a bounded linear operator $T$ from $L^p(\R;\X)$ into $L^p(\R;\Y)$ ($p\in (1 ,\infty)$) is a {\em singular integral operator} if there exists a kernel $K$ such that
\begin{align*} 
& Tf(t) = \int_{\R} K(t,s)f(s) \ud s \\
& \text{for every } f\in L_c^\infty (\R; \X) \text{ and every } t\notin \supp f .
\end{align*} 
Here, $L_c^\infty(\R; X)$ stands for the space of all $X$-valued, essentially bounded, measurable functions with compact support in $\R$. A singular integral operator with kernel $K$ is called a {\em singular integral operator of convolution type} if its kernel is {\em translation-invariant}, that is, it is of the form $K(t,s) = \tilde{K} (t-s)$. We call a singular integral operator a {\em Calder\'on-Zygmund operator} if both the kernel $K$ and the {\em adjoint kernel} $K'$ given by $K'(t,s):=K(s,t)$ ($t$, $s\in \R$) satisfy the classical Lipschitz or {\em second standard condition}, which is the condition that there exists a constant $\delta >0$ such that 
\begin{equation*}
[K]_{D^*_\infty}:=\sup_{t,s,s'\in \R, \atop 2|s-s'|<|t-s|}\frac{|t-s|^{1+\delta}}{|s-s'|^\delta} | K(t,s) - K(t,s')|_{\cL(\X ,\Y )} <\infty , \tag{$D^*_\infty$}
\end{equation*} 
or, in the case of a translation-invariant kernel,
\begin{equation*}
[K]_{H^*_\infty}:=\sup_{t, s\in \R, \atop 2|s|<|t|}\frac{|t|^{1+\delta}}{|s|^\delta} | K(t-s) - K(t)|_{\cL(\X ,\Y )} <\infty . \tag{$H^*_\infty$}
\end{equation*} 

Coifman's inequality, which appears in a general form in \cite[Theorem II]{CoFe74}, says that for every Calder\'on-Zygmund operator $T$, every $p\in(0,\infty)$ and every Muckenhoupt weight $w\in A_\infty(\R)$ there exists a constant $C$, depending on $p$ and the $A_\infty$-constant of $w$, such that 
\begin{equation}\label{Coifman}
 \int_{\R} |Tf|_\Y^p w\; \ud t \leq C \int_{\R} (Mf)^p w\;\ud t \text{ for every } f\in L_c^\infty(\R; X) ,
\end{equation}
$M$ being the classical Hardy-Littlewood maximal operator on $\R$. In other words, all weighted $L^p$-estimates for $M$, for $p\in (0,\infty)$ and $w\in A_\infty(\R)$, are inherited by the operator $T$. Namely, by Muckenhoupt's theorem and Coifman's inequality, a Calder\'on-Zygmund operator extrapolates to a bounded linear operator on $L^p_w(\R ; X)$ for every $p\in(1,\infty)$ and every Muckenhoupt weight $w\in A_p(\R)$. This result is best possible in the sense that there are Calder\'on-Zygmund operators which are not bounded on $L^p_w(\R ; X)$ ($p\in (1,\infty )$) if the weight $w$ does not belong to $A_p(\R)$. At the same time, if the operator $T$ is bounded on $L_w^p(\R; X)$ for some $w\in A_\infty(\R)$  and some $p\in (1,\infty)$, then, by Lebesgue's differentiation theorem, Coifman's inequality (\ref{Coifman}) holds for all $f\in L_w^p(\R;X)$. \\

Subsequently, variants of Coifman's inequality for singular integral operators of convolution type and of nonconvolution type with less regular kernels attracted the attention during the last decades; see, in particular, the pioneering results from Kurtz \& Wheeden \cite{KuWh79}, Rubio de Francia, Ruiz \& Torrea \cite{RuRuTo86}, Alvarez \& P\'erez \cite{AlPe94}. Following the terminology from \cite[Definition 1.1, Part III]{RuRuTo86}, we say that a kernel $K$ satisfies the condition $(D_r )$ ($r\in [1,\infty )$) if 
\[
[K]_{D_r}:=\sup_{s\neq s'}|s-s'|^{\frac{1}{r'}}
\sum_{m=1}^\infty 2^{\frac{m}{r'}} \left(\int_{S_m(s,s')} 
| K(t,s) - K(t,s') |_{\cL(\X ,\Y )}^{r} \ud t\right)^{\frac{1}{r}}<\infty, \tag{$D_r$}
\]
where $S_m(s,s'):=\{ t\in \R: 2^m|s-s'|< |t-s'| \leq 2^{m+1} |s-s'|\}$ ($m\in \N$), and it satisfies the condition $(D_\infty)$ if 
\[
[K]_{D_\infty}:=\sup_{s\neq s'}|s-s'|
\sum_{m=1}^{\infty} 
2^m \sup_{t\in S_m(s,s')} |K(t,s) - K(t,s')|_{\cL(\X ,\Y )} <\infty. \tag{$D_\infty$}
\] 
Moreover, we say that a kernel $K$ satisfies the condition $(D'_r)$ for some $r\in [1,\infty ]$ if its adjoint kernel $K'$ satisfies the condition $(D_r)$, and then we set $[K]_{D'_r}:=[K']_{D_r}$. If a kernel $K$ satisfies the condition $(D_r)$ for some $r\in [1, \infty]$, then it satisfies also the condition $(D_q)$ for every $q\in [1,r]$. Finally, a translation-invariant kernel satisfies the condition $(D_r')$ if and only if it satisfies $(D_r)$, which in the situation of a translation-invariant kernel we denote also by $(H_r)$; see Lorente, Riveros \& de la Torre \cite{LoRiTo05}, but also Kurtz \& Wheeden \cite{KuWh79}. It means that 
\begin{equation}
[K]_{H_r}:=\sup_{s\not= 0}|s|^{\frac{1}{r'}} \sum_{m=1}^\infty 2^{\frac{m}{r'}} \left(\int_{2^m |s|<|t|\leq 2^{m+1}|s|} 
| K(t-s) - K(t) |_{\cL(\X ,\Y )}^{r} \;\ud t\right)^{\frac{1}{r}}<\infty  \tag{$H_r$}
\end{equation}
if $r\in [1,\infty )$, and
\begin{equation}
[K]_{H_\infty}:=  \sup_{s\not= 0}|s|\sum_{m=1}^\infty 2^{m} \, \sup_{2^m |s|<|t|\leq 2^{m+1}|s|} 
| K(t-s) - K(t) |_{\cL(\X ,\Y )}  <\infty. \tag{$H_\infty$}
\end{equation}
Note that the condition $(H_\infty)$ for a function $K \in L^{1}_{loc}(\mathbb R \setminus \{0\}; \cL ({\X}) )$ is weaker than the second standard condition $(H^*_\infty)$.

If a kernel $K$ satisfies the $(H_r)$ condition, then the associated singular integral operator $T$ satisfies Coifman's inequality (\ref{Coifman}) with $M$ replaced by $M_{r'}$ ($r\in (1,\infty ]$, $r':=r/(r-1)$), where $M_rf:=(M(|f|^r))^{1/r}$ for every $f\in L^1_{loc}(\R)$; see \cite[Theorem 1.3, Part II]{RuRuTo86} for the case $r<\infty$ and \cite[Theorem 2.1]{AlPe94} for the case $r=\infty$, and \cite[Part III]{RuRuTo86} for corresponding results in the nonconvolution case. By Martell, P\'erez \& Trujillo-Gonz\'alez \cite[Theorem 3.2]{MaPeTG05}, these variants of Coifman's inequality are sharp in the sense that the maximal operator $M_{r'}$ cannot be replaced by $M_s$ for $s<r'$.

Motivated by the one-sided discrete square function studied in de la Torra \& Torrea \cite{ToTo03} further extensions of the above mentioned results were provided in \cite{LoRiTo05}; see also Lorente, Martell, Riveros \& de la Torre \cite{LMRT08} and the references therein. We point out that the kernel corresponding to this square function satisfies the condition $(H_r)$ for all $r<\infty$, but not $(H_\infty)$. 

In Lorente, Riveros \& de la Torre \cite[Definition 3]{LoRiTo05}, a new scale of H\"ormander conditions $(H_\cA)$, parametrized by Young functions $\cA$, is introduced. For appropriate Young functions $\cA$, the conditions $(H_\cA)$ are intermediate between $(H_\infty)$ and the intersection of all $(H_r)$, $r\in[1,\infty)$.  The corresponding result asserts that if $T$ is a singular operator with translation-invariant kernel satisfying $(H_\cA)$, then for any $p\in (0,\infty)$ and $w\in A_\infty(\R)$, Coifman's inequality $(\ref{Coifman})$ holds with $M$ replaced by $M_{\bar\cA}$, which is associated with the Young function $\cA$ and defined in a similar way as $M_r$; see \cite[Theorem A]{LoRiTo05}. However, note that $Mf\leq C \, M_{\bar\cA}f$ for every Young function $\cA$, and $f\in L_{loc}^1(\R)$. Therefore, such a $M_{\bar\cA}$-variant of Coifman's inequality allows one to deduce the boundedness of $T$ on $L^p_w(\R ; X)$ at most for $p\in(1,\infty)$ and $w\in A_p(\R)$.    

The validity of Coifman's inequality (\ref{Coifman}) for a singular integral operator $T$ with kernel in $\bigcap_{r>1}(H_r)\setminus (H_\infty)$ was left open in \cite{LoRiTo05}; see \cite[Remark 1]{LoRiTo05}. In \cite[Theorem 7]{ChKr14}, the present authors showed that such operators are indeed bounded on $L_w^p(\R;X)$ for every $p\in (1,\infty )$ and $w\in A_p(\R)$. The latter result is even true in the more general case of nonconvolution kernels, under asymmetric regularity assumptions with respect to the variables, and in the context of general rearrangement invariant Banach function spaces. \\

In addition, in \cite[Section 3]{LoRiTo05} the question was raised whether it is possible to improve the above results in the case of one-sided singular integral operators, that is, with translation-invariant kernels supported in the half-line $\R_- := (-\infty ,0]$ or $\R_+ := [0,\infty )$. In particular, \cite[Theorem 3]{LoRiTo05} says that if a kernel $K$ satisfies the $(H_{\cA})$ condition and if $\supp K\subseteq \R_-$, then for every $p\in(0, \infty)$ and $w\in A^+_\infty(\R)$ there exists a constant $C$ such that for a singular integral operator associated with $K$
\begin{equation}\label{onesidedCF}
 \int_\R |Tf|_\Y^p w \, \ud t\leq C \, \int_\R \left(M^+_{\bar\cA}f\right)^p w \, \ud t\quad \text{for every } f\in L_c^\infty(\R;X).
\end{equation}
Recall that $A^\pm_\infty(\R)$ is the union of the classes $A_p^\pm(\R)$, $p\in[1,\infty)$, introduced by Sawyer \cite{Sw86} to characterise the weighted $L^p$-estimates for the one-sided Hardy-Littlewood operators $M^\pm$, which are are defined as follows:
\begin{equation}\label{M op}
 M^-f(t)  := \sup_{h>0} \frac{1}{h} \int_{t-h}^{t} |f| \,\ud s, \quad \text{ and }\quad 
M^+f(t)  := \sup_{h>0} \frac{1}{h} \int_{t}^{t+h} |f| \, \ud s\phantom{,}
\end{equation}
for every $f\in L_{loc}^1(\R)$ and $t\in\R$. Recall also the corresponding \emph{one-sided $A_p$ conditions} introduced by Sawyer. For example, we say that a weight $w$ satisfies the $A_p^+$ condition on $\R$, and write $w\in A^+_p(\R)$, if
\[
\sup_{a\in\R, \, h>0} \frac{1}{h} \int^a_{a-h} w\, \ud t\, \left(\frac{1}{h} \int^{a+h}_a w^{1-p'}\ud t\right)^{p-1} < \infty.
\]
The class $A_p^-(\R)$ is defined analoguously. Sawyer proved the following analogue of Muckenhoupt's theorem, namely that for every $p\in (1,\infty )$ and for any weight $w$ on $\R$, $M^\pm$ is bounded on $L^p_w (\R )$ if and only if $w\in A_p^\pm(\R)$; see \cite[Theorem 1]{Sw86}. Note that the class $A^+_p (\R)$ is bigger than the class $A_p (\R)$. Indeed, following Sawyer's remark, the product $w=\phi v$ of a nondecreasing function $\phi$ and a Muckenhoupt weight $v\in A_p(\R)$ belongs to $A_p^+(\R)$. In particular, $A_p^+(\R)$, contains all positive nondecreasing functions on $\R$. Furthermore, $A^\pm_\infty(\R)\subseteq L^1_{loc}(\R)$. It is easily seen that for every Young function $\cA$, $M^+f \leq M_{\bar \cA}^+f$ $(f\in L_{loc}^1(\R))$. Therefore, by Sawyer's result, the one-sided Coifman inequality (\ref{onesidedCF}) controls the boundedness of $T$ on $L^p_w(\R ; X)$  at most for $p\in(1,\infty)$ and $w\in A_p^+(\R)$.\\

In this article we continue the above presented line of researches on weighted estimates for the one-sided Hardy-Littlewood maximal operators on the half-line, and for singular integral operators on $L^p (\R_\pm ;X)$ with kernels which satisfy the $(H_r)$ condition for all $r\in[1,\infty)$ and which are supported on a half-line. The main result of the article is Theorem \ref{Boyd th}.
 
The main ingredients of the proof of Theorem \ref{Boyd th} may be of independent interest. In particular, in Section \ref{sec.Muckenhoupt} we show that the operator $M^+$ is bounded on $L^p_w(\R_- )$ if and only if the weight $w$ satisfies a Sawyer type condition $A^+_p$ restricted to $\R_-$; see Definition \ref{def.weight} and Theorem \ref{thm.restSawyer}. For this counterpart of Sawyer's theorem, we adapt the approach presented in Martin-Reyes \cite{MR93}. In particular, we show that the classes $A_p^+(\R_-)$, $p\in(1,\infty)$, possess the openness property; see Lemma \ref{lem.openness}. 

The other main ingredient is an analogue of the Fefferman-Stein inequality (Lemma \ref{FSineq}), which allows us to provide integral estimates for singular operators; see Theorem \ref{thm.exsio}. In Section \ref{sec.r.i.extrap}, we provide a counterpart of the Lorenz-Shimogaki theorem for the operator $M^+$; see Proposition \ref{LSanalogue}. Then, by means of Rubio de Francia's iteration algorithm we prove the main result, Theorem \ref{Boyd th}. It extends also the result by Curbera, Garcia-Cuerva, Martell \& P\'erez on extrapolation of Calder\'on-Zygmund operators for weighted rearrangement invariant Banach function spaces; see, for example, \cite[Section 3.6]{CrMaPe11} and corresponding references therein.

In the proofs of the main ingredients of Theorem \ref{Boyd th} we reproduce standard techniques from the Calder\'on-Zygmund theory, especially developed in Benedek, Calder\'on \& Panzone \cite{BeCaPa62}, Calder\'on \cite{Ca66}, Martin-Reyes \cite{MR93}, Martin-Reyes, Pick \& de la Torre \cite{MRPiTo93}, Martin-Reyes \& de la Torre \cite{MRTo94}, Muckenhoupt \cite{Mu72}, Rubio de Francia \cite{Ru84}, Rubio de Francia, Ruiz \& Torrea \cite{RuRuTo86}, and Sawyer \cite{Sw86}.

In Section \ref{sec.apply}, we apply our results to study the extrapolation of $L^p$-maximal regularity for the first order, second order and fractional order Cauchy problems as well as for Volterra equations. Theorems \ref{thm.first.order} and \ref{thm.mrextra.second} provide generalisations of recent results from Auscher \& Axelsson \cite{AuAx11a}, Chill \& Fiorenza \cite{ChFi14} and the present authors \cite{ChKr14}; see also Pr\"uss \& Simonett \cite{PrSi04} and Haak \& Kunstmann \cite{HaKu07} for earlier results of this type. To keep the presentation of the main result, Theorem \ref{Boyd th}, more transparent, more detailed information on these applications is included in Section \ref{sec.apply}.

\section{The one-sided Muckenhoupt $A_p$-condition on a half-line}\label{sec.Muckenhoupt}

In this section we prove a counterpart of Sawyer's theorem for the one-sided Hardy-Littlewood operators restricted to functions supported on a half-line; see Theorem \ref{thm.restSawyer} below. Set 
\begin{align*}
M^\pm_- f & := \chi_{\R_-}\cdot M^\pm f \quad  \text{ for every } f\in L^1_{loc}(\R_-),
\textrm{ and }\\ 
M^\pm_+ f &:=\chi_{\R_+}\cdot M^\pm f \quad  \text{ for every } f\in L^1_{loc}(\R_+),
\end{align*}  
where $M^+$ and $M^-$ denote the one-sided Hardy-Littlewood operators (see (\ref{M op}) for the definition), and $L^1_{loc} (\R_+)$ and $L^1_{loc} (\R_-)$ are understood as subspaces of $L^1_{loc} (\R)$ in the natural way. We point out that the maximal function as originally defined by Hardy and Littlewood corresponds to the operator $M^-_+$. 
With each of these operators we associate a class of weights, the definition of which is motivated by the argument used in the proof of the first part of Lemma \ref{lem.weak-estim} below. Throughout, a measurable function $w:\R_\pm\rightarrow \R_+$ with $w(t)>0$ for almost every $t\in \R_\pm$ is called a \emph{weight} on $\R_\pm$.

\begin{definition} \label{def.weight} 
Let $w$ be a weight on $\R_-$.
\begin{itemize}
\item [(a)] For $p\in (1,\infty)$ we say that $w$ satisfies {\em Sawyer's $A_p^+$ condition} (resp. {\em Sawyer's $A_p^-$ condition}) on $\R_-$, and we write $w\in A^+_p(\R_-)$ (resp. $w\in A^-_p (\R_-)$), if
\begin{align}
\label{A+-}[w]_{A_p^+(\R_-)} & :=\sup_{a<b<c\leq 0}\frac{1}{(c-a)^p}
 \int_a^b w\,\ud t \, \left( \int_b^c w^{1-p'}\,\ud t \right)^{p-1}<\infty \\
\nonumber (\text{resp. } [w]_{A_p^-(\R_-)} & :=\sup_{a<b<c\leq 0}\frac{1}{(c-a)^p}
 \int_b^c w\, \ud t \, \left( \int_a^b w^{1-p'}\, \ud t \right)^{p-1}<\infty ).
\end{align}
 \item [(b)] We say that $w$ belongs to $A_1^+(\R_-)$ (resp. $A_1^- (\R_-)$), if 
\begin{align*} 
[w]_{A_1^+(\R_-)} & :=\|M^- w/ w\|_{L^\infty(\R_-)}<\infty \\
(\text{resp. } [w]_{A_1^-(\R_-)} & :=\|M^+ w/ w\|_{L^\infty(\R_-)}<\infty ).
\end{align*}
 \item [(c)] We say that $w$ belongs to $A_\infty^+(\R_-)$ (resp. $A_\infty^-(\R_-)$), if there exist constants $C$, $\delta>0$ such that for every $a<b<c\leq 0$ and every measurable set $S\subseteq [b,c]$ (resp. $S\subseteq [a,b]$) 
\begin{align} 
\label{Ainfty-} \frac{|S|}{c-a} & \leq C \left(\frac{w(S)}{w(a,b)} \right)^{\delta} \\ 
\nonumber (\text{resp. } \frac{|S|}{c-a} & \leq C \left(\frac{w(S)}{w(b,c)} \right)^{\delta} ), 
\end{align}
Here and subsequently, for every measurable set $A\subseteq \R_-$, we write $w(A)$ for $\int_A w\,\ud t$.
\item [(d)] Moreover, for $p\in[1,\infty)$ we say that a weight $w$ on $\R_+$ belongs to  $A_p^\pm(\R_+)$ if and only if $w(-\cdot)\in A_p^\mp(\R_-)$, and then we set $[w]_{A_p^\pm(\R_+)} := [w(-\cdot)]_{A_p^\mp(\R_-)}$. For $p=\infty$, 
we say that $w$ belongs to $A_\infty^+(\R_+)$ (resp. $A_\infty^-(\R_+)$) , if there exist constants $C$, $\delta>0$ such that for every $0\leq a<b<c$ and every measurable set $S\subseteq [b,c]$ (resp. $S\subseteq [a,b]$) 
\begin{align} 
\label{Ainfty+}  \frac{|S|}{c-a} & \leq C \left(\frac{w(S)}{w(a,b)} \right)^{\delta} \\ 
\nonumber (\text{resp. }\frac{|S|}{c-a} & \leq C \left(\frac{w(S)}{w(b,c)} \right)^{\delta} ) .
\end{align}
\end{itemize} 
\end{definition}

The classes $A_p^\pm(\R_+)$ ($p\in[1,\infty)$) may of course be defined directly, too, by using expressions which are symmetric to those used in the definition of the classes $A^\pm_p(\R_-)$. For instance, a weight $w$ on $\R_+$ belongs to $A^-_p(\R_+)$ ($1<p<\infty$) if and only if 
\begin{equation*}
 [w]_{A_p^+(\R_+)}  :=\sup_{0 \leq a<b<c}\frac{1}{(c-a)^p}
 \int_b^c w\,\ud t \, \left( \int_a^b w^{1-p'}\,\ud t \right)^{p-1}<\infty.
\end{equation*}
Since according to our definition any weight on a half-line is almost everywhere finite and (strictly) positive, any weight $w$ in $A^\pm_p(\R_-)$ (resp. $A_p^\pm(\R_+)$) is locally integrable on $(-\infty, 0)$ (resp. $(0,\infty)$).

\begin{remark} \label{rem.weights}
If $w: (-\infty ,0) \to \R_+$ is an increasing function and if $v\in A_p^+ (\R_-)$ ($p\in [1,\infty )$), then the product $wv$ belongs to $A_p^+ (\R_-)$, too. This follows easily from the definition. In particular, every increasing function $w: (-\infty ,0)\to\R_+$ belongs to $A_p^+ (\R_- )$ for every $p\in [1,\infty )$. Similarly, every decreasing function $w: (0,\infty )\to\R_+$ belongs to $A_p^- (\R_+)$ for every $p\in [1,\infty )$. In particular, the decreasing power weights given by $w(t) = t^\beta$ belong to $A_p^- (\R_+ )$ for every $p\in [1,\infty )$ and every $\beta\in (-\infty ,0]$. This is in contrast to the Muckenhoupt power weights where necessarily $\beta >-1$. 
\end{remark}

The following theorem provides a counterpart for the operators $M^\pm_-$ and $M^\pm_+$ of Sawyer's result \cite[Theorem 1]{Sw86}.

\begin{theorem} \label{thm.restSawyer} 
Let $1<p<\infty$ and $w$ be a weight on $\R_-$ (resp. $\R_+$).
Then the Hardy-Littlewood operators $M^\pm_-$ (resp. $M^\pm_+$) are bounded on $L^p_w(\R_-)$ (resp.  $L^p_w(\R_+)$) if and only if $w\in A^\pm_p(\R_-)$ (resp. $w\in A^\pm_p(\R_+)$). 
\end{theorem}

We start with some preliminary observations on the pattern of the proof of Theorem \ref{thm.restSawyer} provided below.
Namely, the proof of this result follows the idea of an alternative proof of Sawyer's theorem given by Martin-Reyes in \cite{MR93}. By means of the classical Marcinkiewicz theorem, Theorem \ref{thm.restSawyer} follows from the fact that Sawyer's $A^\pm_p$ conditions on $\R_-$ (resp. $\R_+$) characterize the weights $w$ for which the operators 
$M^\pm_-$ (resp. $M^\pm_+$) are of weak type $(p,p)$ with respect to $(\R_-, w\,\ud t)$ (resp. $(\R_+, w\, \ud t)$), and the fact that the classes $A^\pm_p(\R_-)$ (resp. $A^\pm_p(\R_+)$) possess the openness property; see Lemmas \ref{lem.weak-estim} and \ref{lem.openness} below.

Moreover, one can easily check that 
\begin{align}\label{M-red.step}
 M^\pm_-f(t) & = M^\mp_+(f(-\cdot))(-t), \quad (t\in\R_- , \;f\in L^p_w(\R_- ),\; w\in A^\pm_p(\R_-)), 
\end{align}
These formulas allow one to reduce the proof of Theorem \ref{thm.restSawyer} (and of Lemmas \ref{lem.weak-estim} and \ref{lem.openness}) to one of the following pairs of operators; e.g., either $M^+_-$ and $M^+_+$, or $M^-_-$ and $M^+_-$, or also $M^-_+$ and $M^+_+$. 
The direct proofs for each of these pairs follow the same ideas, but their presentations differ in several details. Since the proofs of relevant results from \cite{MR93, Sw86} are formulated for the operator $M^+$ and the Sawyer $A^+_p(\R)$ classes, and since the presentation of the proof for the operators $M^+_\pm$ and the classes $A^+_p(\R_\pm)$, except for some details, follows essentially that of \cite[Theorem 1 and Proposition 3]{MR93}, we provide below the direct proof just for this pair. The same remark applies to the proofs in Section 3. This allows us to refer the reader to \cite{MR93} for common ingredients, 
and to provide only main supplementary observations which should be made. Then, the corresponding statements for $M^-_\pm$ and $A^-_p(\R_\pm)$ follow simply from 
(\ref{M-red.step}). 

\begin{lemma}\label{lem.weak-estim} 
Let $1< p < \infty$, and $w$ be a weight on $\R_\pm$. Then the operator $M^+_\pm$ is of weak type $(p,p)$ with respect to $(\R_\pm, w\,\ud
t)$, that is, 
\begin{equation} \label{equ1}
 w\left(\{ t\in \R_\pm: M^+_\pm f(t) >\lambda \}\right)\leq \frac{C}{\lambda^p}
 \int_{\R_\pm}|f|^p w\;\ud t \quad (\lambda >0, \; f\in L^p_w(\R_\pm ))
\end{equation}
if and only if $w\in A^+_p(\R_\pm)$. Moreover, if $w\in A^+_p(\R_\pm)$, then we can take $C=4^p[w]_{A_p^+(\R_\pm)}$ in (\ref{equ1}). The analoguous statement holds for the operators $M_\pm^-$. 
\end{lemma}

\begin{proof} We present the proof only for the operator $M^+_-$. The proof for $M^+_+$ follows exactly the lines of that for $M^+_-$.  Then, the proof for the operators $M^-_\pm$ follows easily by the symmetry argument from (\ref{M-red.step}).

{\em Necessity.} This part of the proof is standard, but we provide it for completeness.
Assume that $w$ is a weight on $\R_-$ such that $M^+_-$ is of weak  type $(p,p)$. Fix $ a < b< c\leq 0$ and $f\in L^p_w(\R_- )$ such that 
$0 <\int_a^b |f| \ud t <\infty$. Note that for every $\lambda >0$ with
$\lambda < \frac{1}{c-a}\int_b^c |f| \ud t$ we have that $[a,b] \subseteq \{ s\in \R_-: M^+_-f(s) > \lambda\}$. 
Consequently, since $\lambda< \frac{1}{c-a}\int_b^c |f| \ud t$ is arbitrary, we get from \eqref{equ1} that $w$ is locally integrable on $(-\infty, 0)$ and  
\[
\int_a^b w\, \ud t \leq \frac{(c-a)^p}{(\int_b^c |f| \ud t)^p} \int_{\R_-} |f|^p w\,\ud t.
\]
Substituting $f:= w^{1-p'}\chi_{A_\varepsilon}$, where $A_\varepsilon:=\{s\in [b,c]: w(s)\geq \varepsilon \}$ for $\varepsilon>0$ small enough, by a limiting argument we find that $w$ is in $A_p^+(\R_-)$.

{\em Sufficiency.} This part of the proof of Lemma \ref{lem.weak-estim} can be obtained by a simple adaptation of the proof of \cite[Theorem 1, $(b)\Rightarrow(a)$]{MR93}. 

Fix $w\in A^+_p(\R_-)$ $(1<p< \infty)$. 
Since $w$ is locally integrable on $(-\infty, 0)$,  it is sufficient to prove (\ref{equ1}) for a nonnegative bounded function $f$ with compact support in $(-\infty, 0)$. 
Since $M^+_-f$ is lower semicontinuous, for every $\lambda>0$ with $\{ t\in \R_-: M^+_-f(t) > \lambda \} \neq \emptyset$, there exists $N\in\N\cup \{\infty\}$ such that
\[
\{ t\in \R_-: M^+_-f(t) >\lambda \}  = \bigcup_{j=0}^N I_j
\] 
with $I_j=(a_j,b_j)$ for all $j=0$, $\dots$, $N.$ All intervals $I_j$ are pairwise disjoint, and
\begin{equation}\label{lem1ineq1}
\lambda \leq \frac{1}{b_j-s}\int_s^{b_j} f \ud t \quad \text{for every } s\in [a_j, b_j), \, j=0, \, \dots , \, N.
\end{equation}
For a fixed $j\in\{0$, $\dots$, $N\}$ we set $(a,b):=I_j$. It is sufficient to show that $\int_a^b w\,\ud t\leq C \lambda^{-p}\int_a^b f^p  w\,\ud t $ for some constant $C$ independent of $I_j$, $\lambda$, and $f$. 
Let $(s_k)_{k\geq 0}$ be the increasing sequence defined by the following condition:
\[
s_0:=a\quad \text{and} \quad \int_{s_{k-1}}^{s_k} f\, \ud t = \int_{s_k}^b f\,\ud t \;\; (k\geq 1).
\]
Note that $(a,b)=\bigcup_{k\geq 0} (s_k, s_{k+1}]$, and 
$\int_{s_{k-1}}^b f\,\ud t= 4 \int_{s_{k}}^{s_{k+1}} f \,\ud t$ ($k\geq 1$). Consequently, by \eqref{lem1ineq1},
\[
\lambda \leq \frac{1}{b-s_{k-1}} \int_{s_{k-1}}^b f\,\ud t=  \frac{4}{b-s_{k-1}}\int_{s_{k}}^{s_{k+1}} f \,\ud t \quad (k\geq 1).
\]
Since $w$ is locally integrable on $(-\infty,0)$, by H\"older's inequality and \eqref{A+-} we obtain
\begin{align*}
 \int_{s_{k-1}}^{s_k}w\,\ud t & \leq \frac{4^p}{\lambda^p(b - s_{k-1})^p} \int_{s_{k-1}}^{s_k} 
 w\,\ud t \left(\int_{s_{k}}^{s_{k+1}} w^{1-p'}\,\ud t \right)^{p-1} \int_{s_k}^{s_{k+1}} f^p w\,\ud t\\
 & \leq \frac{4^p[w]_{A_p^+(\R_-)}}{\lambda^p}\left(\frac{s_{k+1} - s_{k-1}}{b - s_{k-1}}\right)^p  \int_{s_k}^{s_{k+1}} f^p w\,\ud t\leq \frac{4^p[w]_{A_p^+(\R_-)}}{\lambda^p} \int_{s_k}^{s_{k+1}} f^p 
 w\,\ud t.
\end{align*}
Summing over $k\geq 1$ we get: 
\[
\int_{a}^{b} w\, \ud t \leq  \frac{4^p[w]_{A_p^+(\R_-)}}{\lambda^p} \int_{a}^{{b}} f^p w\, \ud t.
\]
This completes the proof.
\end{proof}

The second lemma asserts that the classes $A_p^+ (\R_\pm)$ and $A_p^- (\R_\pm)$ possess the openness property.

\begin{lemma}\label{lem.openness}
 Let $1<p<\infty$ and $w\in A_p^+ (\R_\pm)$. Then,  
\[
\inf\{q>1: w\in A_q^+ (\R_\pm)\} < p . 
\]
The corresponding result holds for the classes $A_p^- (\R_\pm)$.
\end{lemma}

Since $w\in A_p^- (\R_\pm )$ if and only if $w(- \cdot)\in A_p^+ (\R_\mp)$, it is sufficient to prove Lemma \ref{lem.openness} for the classes $A_p^+ (\R_\pm)$. The proof in this case can be obtained by a simple adaptation of the proof of \cite[Proposition 3]{MR93}. We give some details for the convenience of the reader and for our further purposes; see, for example, Corollary \ref{unif bound} below.

\begin{proof}
We consider only the case of the classes $A^+_p(\R_-)$. The same arguments apply to $A^+_p(\R_+)$.
Fix $p\in (1,\infty)$ and $w\in A_p^+(\R_-)$. By H\"older's inequality one can show that $A^+_q(\R_-)\subseteq A^+_p(\R_-)$ for every $q\in (1,p)$. Therefore, it is sufficient to show that there exist $q<p$, $C>0$ such that:
\begin{equation}\label{A++s}
\sup_{a<b<c<0}\frac{1}{(c-a)^q}
\left( \int_a^b w \; \ud t\right) \, \left( \int_b^cw^{1-q'} \; \ud t \right)^{q-1}\leq C .
\end{equation}

The proof of \eqref{A++s} is based on a variant of the reverse H\"older inequality for the weight 
$\sigma:=w^{1-p'}$, which says that there exist constants  $C$, $\delta>0$ such that
\begin{equation} \label{weakrH}
\frac{1}{c-b}\int_b^c \sigma^{1+\delta} \ud t \leq C \left( M^+_-(\sigma \chi_{(b,c)}) (b)\right)^{1+\delta} 
\end{equation}
for every $b<c<0$. To prove it, fix $I=(b,c)$ ($b<c<0$) and  note that $\sigma\in A_{p'}^-(\R_-)$. Additionally, assume that $\sigma^{1+\delta}$ is locally integrable on $(-\infty, 0)$ for any $\delta >0$.
 Following the idea of the proof of \cite[Lemma 5]{MR93}, set $\lla_0:=M^+_-(\sigma\chi_I)(b)$
  and $S_\lla:= \{t\in \R_-: M^+_-(\sigma\chi_I)(t)>\lambda\}$ ($\lambda>\lambda_0$). 
 For $\lla>\lla_0$ with $S_\lla\neq \emptyset$, similarly as in the proof of Lemma \ref{lem.weak-estim}, we can write $S_\lla =\bigcup_j I_j$. Note that   $S_\lla \subseteq I$. 
Furthermore,  there exist constants $\alpha$, $\beta>0$ such that 
\begin{equation}\label{lem2inq1}
 |\{ t\in I_j: \sigma(t)>\beta \lambda \}| > \alpha|I_j| \quad \text{for all } j .
\end{equation} 
To see it, similarly as in the proof of Lemma \ref{lem.weak-estim}, for a fixed interval $I_j$, define an increasing sequence $(s_k(j))=(s_k)$ by the following condition:
\[
s_0:=a_j, \quad  s_k<s_{k+1}, \quad \int_{s_k}^{s_{k+1}} \sigma\, \ud t = \int^{b_j}_{s_{k+1}} \sigma\, \ud t.
\]
Therefore, by (\ref{lem1ineq1}) with $f$ replaced by $\sigma\chi_I$, we have that 
\[
\lambda \leq \frac{1}{b_j-{s_{k}}}\int_{s_{k}}^{b_j} \sigma\; \ud t = \frac{2}{b_j-{s_{k}}}\int_{s_{k+1}}^{b_j} \sigma\; \ud t = \frac{4}{b_j-{s_{k}}}\int_{s_{k+1}}^{s_{k+2}}\sigma\; \ud t\quad (k\geq0).
\] 
Set $E_k:=E_{k, \beta}:=\left\{t\in (s_k, s_{k+1}]: \sigma(t) \leq \beta \frac{4}{b_j-{s_{k}}}\int_{s_{k+1}}^{s_{k+2}}\sigma\; \ud t \right\}$ ($\beta>0$). 

Since $\sigma\in A^-_{p'}(\R_-)$ we get:
\begin{align*}
 \left( \frac{|E_k|}{s_{k+2} - s_{k}}\right)^{p' -1} & =
 \frac{4}{{b_j} - s_{k}} \left( \int_{s_{k+1}}^{s_{k+2}}\sigma\; \ud t\right)\, \left( \frac{1}{s_{k+2} - s_{k}}\int_{E_k} \left( \frac{4}{{b_j} - s_{k}}\int_{s_{k+1}}^{s_{k+2}}\sigma\, \ud t\right)^{1-p} \ud \tau \right)^{p' -1}\\
 &\leq \frac{4\beta }{{b_j} - s_{k}} \left(\int_{s_{k+1}}^{s_{k+2}}\sigma\; \ud t\right)\, \left( \frac{1}{s_{k+2} - s_{k}}\int^{s_{k+1}}_{s_{k}} \sigma ^{1-p}\; \ud t\right)^{p' -1}\\
 &\leq 4 \beta [\sigma]_{A_{p'}^-(\R_-)}.
\end{align*}
Since $\{t\in (s_k, s_{k+1}]: \sigma(t) \leq \beta \lambda\}\subseteq E_k$, we obtain that:
\begin{align*}
 |\{t\in I_j: \sigma(t) \leq \beta \lambda\}| & \leq \sum_{k\geq 0} |E_k| \leq (4\beta [\sigma]_{A_{p'}^-(\R_-)})^{p-1} \sum_{k\geq 0} (s_{k+2} - s_{k}) \\
& \leq 2(4\beta [\sigma]_{A_{p'}^-(\R_-)})^{p-1}.
\end{align*}
 This gives, 
\begin{equation}\label{alpha}
 |\{t\in I_j: \sigma(t) > \beta \lambda\}|\geq |I_j| \left(1- 2(4\beta [\sigma]_{A_{p'}^-(\R_-)})^{p-1}\right).
\end{equation}
By taking $\beta$ small enough, we obtain (\ref{lem2inq1}) for $\alpha:=1- 2(4\beta [\sigma]_{A_{p'}^-(\R_-)})^{p-1}$. 
Consequently,  by (\ref{lem1ineq1}) (for $f=\sigma\chi_I$, and note also that $a_j\notin S_\lla$) and  (\ref{lem2inq1}), we have that 
\begin{align*}
\sigma\left(S_\lla\right) 
 & = \sum_j \sigma(I_j) =  \lambda \sum_j |I_j| \\
 &\leq \alpha^{-1} \lambda |\{s\in I: \sigma(s) > \beta \lambda\}|. 
\end{align*}

To continue the proof of (\ref{weakrH}), by Lebesgue's differentiation theorem, we get 
\[
\sigma\left(\{ s\in I: \sigma(s) >\lambda \} \right) \leq\sigma (S_\lla)\quad 
\textrm{for every } \lla>\lla_0.
\]
Now, on the one hand,
\begin{align*}
 \int_{\lambda_0}^\infty \lla^{\delta -1} \sigma(\{ s\in I: \sigma(s)> \lla\})\,\ud \lla
 & \leq  \int_{\lambda_0}^\infty \frac{\lla^\delta}{\alpha} 
 \left| \{s\in I: \sigma(s) > \beta \lambda\} \right| \,\ud \lla\\
 & \leq  \frac{1}{(1+\delta)\alpha \beta^{1+\delta}} \int_I \sigma^{1+\delta} \,\ud t.
\end{align*}
And, on the other hand, 
\begin{align*}
 \int_{\lambda_0}^\infty \lla^{\delta -1} \sigma(\{ s\in I: \sigma(s)> \lla\})\,\ud \lla
 & = \int_{\{s\in I: \sigma(s) >\lla_0\}} \sigma(t) \int_{\lla_0}^{\sigma(t)} \lla^{\delta -1} \,\ud \lla \,\ud t \\
 & = \int_{\{s\in I: \sigma(s) >\lla_0\}} \sigma(t) \left( \frac{\sigma(t)^{\delta}}{\delta} - \frac{\lla_0^\delta}{\delta} \right) \,\ud t \\
 & \geq  \frac{1}{\delta}\int_I \sigma^{1+\delta}(t) \,\ud t - \frac{\lla_0^\delta}{\delta} \int_I 
 \sigma(t) \,\ud t.
\end{align*}
Consequently, we get
\begin{equation}\label{delta}
 \left(\frac{1}{\delta} - \frac{1}{1+\delta\alpha\beta^{1+\delta}}\right) 
\int_I \sigma^{1+\delta}\; \ud t \leq \frac{\lambda^\delta_0}{\delta}\int_I \sigma\; \ud t
\end{equation}
for every $\delta>0$. Therefore, by taking $\delta$ small enough, we find (\ref{weakrH}). 
To remove our additional assumption on the integrability of $\sigma^{1+\delta}$, set 
 $\sigma_k:=\inf(\sigma, k)$ ($k>0$). It is easy to check that $\sigma_k$ belongs to $A^-_{p'}(\R_-)$ with $[\sigma_k]_{A^-_{p'}(\R_-)}\leq 2^p(1+ [\sigma]_{A^-_{p'}(\R_-)})$ ($k>0$).
Then, by an appropriate choice of the constants $\alpha$, $\beta$ and $\delta$, and by a standard limiting argument, we get the general case.

We are now in a position to show \eqref{A++s} with $q:=\frac{p+\delta}{1+\delta}$ for any $\delta>0$ such that \eqref{weakrH} holds. Fix $a<b<c<0$, and note that $\sigma^{1+\delta}$ is integrable over $(a,c)$. Following \cite{MR93}, define a decreasing sequence $(s_k)_{k=0}^N$ as follows: 
\begin{align*}
& s_0:=b> s_1 > ...> s_N \geq a=:s_{N+1} , \\
& \int_{s_k}^c \sigma^{1+\delta} \,\ud t=2^k\int_b^c \sigma^{1+\delta} \,\ud t \quad \textrm{ if } k=0,\dots ,N, \quad \textrm{ and} \\
& \int_a^{s_N} \sigma^{1+\delta} \,\ud t < 2^N\int_b^c \sigma^{1+\delta} \,\ud t. 
\end{align*}
In particular, 
$\int_{s_{N+1}}^c \sigma^{1+\delta} \ud t \leq 2^{N+1} \int_b^c\sigma^{1+\delta} \ud t.$
Therefore, by \eqref{weakrH} we get
\begin{align*}
 \int_{a}^b w\, \ud t \left(\frac{1}{c-a}\int_b^c \sigma^{1+\delta} \, \ud t \right)^q
& = \sum_{k=0}^N 2^{-kq} \int_{s_{k+1}}^{s_k} w\, \ud t \left(\frac{1}{c-a}\int_{s_k}^c \sigma^{1+\delta} \, \ud t \right)^q \\
& \leq \sum_{k=0}^N 2^{-kq} \int_{s_{k+1}}^{s_k} w(t)  \left(\frac{1}{c-t}\int_t^c \sigma^{1+\delta} \, \ud s  \right)^q    \ud t \\
& \leq C\, \sum_{k=0}^N 2^{-kq} \int_{s_{k+1}}^{s_k}  \left( M^+_-(\sigma \chi_{(s_{k+1}, c)})(t) \right)^{p+\delta}  w(t) \,\ud t.
\end{align*}
Since the operator $M^+_-$ is bounded on $L^\infty_w (\R_+ )$ and $w\in A^+_{p+\delta}(\R_-)$, by Lemma \ref{lem.weak-estim} and the Marcinkiewicz interpolation theorem, we get that 
$M^+_-$ is bounded on $L^{p+\delta}_w(\R_- )$. Hence,  
\begin{eqnarray*}
 \int_{a}^b w\, \ud t \left(\frac{1}{c-a}\int_b^c \sigma^{1+\delta} \ud t \right)^q
&\leq & C \sum_{k=0}^N 2^{-kq} \int_{s_{k+1}}^{c} \sigma^{1+\delta} \ud t \\
&\leq &  C \sum_{k=0}^N \frac{2^{k+1}}{2^{qk}} \int_{b}^{c} \sigma^{1+\delta} \ud t. \\
\end{eqnarray*}
Since $q>1$, the proof is complete.
\end{proof}

\begin{proof}[Proof of Theorem \ref{thm.restSawyer}]
The necessity follows immediately from Lemma \ref{lem.weak-estim}. 

Conversely, let $w\in A^+_p(\R_\pm)$ for some $p\in (1,\infty)$. By Lemma \ref{lem.openness}, there exists $q<p$ such that $w\in A^+_q(\R_\pm)$.  Since the operator $M^+_\pm$ is bounded on $L^\infty_w(\R_\pm )$, by Lemma \ref{lem.weak-estim} and the Marcinkiewicz interpolation theorem, we get the boundedness of $M^+_\pm$ on $L^p_w(\R_\pm )$. 

By the symmetry argument from the remark following Theorem \ref{thm.restSawyer} (see the equality \eqref{M-red.step}), we get the corresponding statement for $M^-_\pm$, and the proof is complete.
\end{proof}

For further references, we point out the following observation which is crucial for the proof of Theorem \ref{Boyd th} in Section \ref{sec.r.i.extrap}. 

\begin{corollary} \label{unif bound}
Let $1<p<\infty$, and let $\cF$ be a subset of $A_p^+(\R_\pm)$, such that 
\[
\sup_{w\in \cF}[w]_{A_{p}^+ (\R_\pm)} <\infty .
\]
Then, there exists $q<p$ such that $\cF\subseteq A_q^+(\R_\pm)$ and 
\[
\sup_{w\in \cF}[w]_{A_{q}^+(\R_\pm)} <\infty .
\]
In particular, $\sup_{w\in \cF}\|M^+_\pm\|_{p,w}<\infty$.

The corresponding results hold for the classes $A^-_p(\R_\pm)$ and the operators $M^-_\pm$.
\end{corollary}

To see it, the reader should have in mind the inequalities (\ref{alpha}) and (\ref{delta}) which lead to the choice of the constants $\alpha, \beta,$ and $\delta$ in the proof of Lemma \ref{lem.openness}, and the explicit expression of the constants involved in the formulation of the Marcinkiewicz interpolation theorem; see e.g. the formulation in \cite[Theorem 1.3.2, Chapter I]{Gr08}.

\section{The Coifman type inequality}\label{sec.Coifman}

We now turn to a variant of Coifman's inequality for the class $A^+_\infty(\R_-)$, the operator $M^+_-$, and an appropriate class of singular integral operators. 

We first extend the notion of singular integral operators. Namely, we say that a bounded linear operator $T$ from $L^p(\R_\pm;\X)$ into $L^p(\R_\pm;\Y)$ ($p\in(1,\infty)$) is a singular integral operator if there exists a kernel $K$ such that 
\begin{align*}
& Tf(t) = \int_{\R_\pm} K(t,s)f(s) \; \ud s \\
& \text{for every } f\in L^\infty_c (\R_\pm ; \X) \text{ and every } t\in \R_\pm\setminus \supp f.
\end{align*}
Note that in this case the values $K(t,s)$ for $(t,s)\notin \R_\pm\times \R_\pm$ are immaterial, and we can assume that $K$ is defined on $\R_\pm\times \R_\pm$.

Furthermore, for kernels supported on $\{(t,s)\in\R_-\times\R_- : t<s \}$, we relax the conditions $(D_r)$ and $(D_r')$ $(r\in [1,\infty])$ to the following ones:
\[
[K]_{D_{r,-}}:=\sup_{s<0, h>0}h^{\frac{1}{r'}}
\sum_{m=1}^\infty 2^{\frac{m}{r'}} \left(\int_{I_m(s,h)} 
| K(t,s) - K(t,s-h) |_{\cL(\X ,\Y )}^{r} \,\ud t\right)^{\frac{1}{r}}<\infty, \tag{$D_{r,-}$}
\]
\[
[K]_{D_{r,-}'}:=\sup_{t<0, h>0}h^{\frac{1}{r'}}
\sum_{m=1}^\infty 2^{\frac{m}{r'}} \left(\int_{J_m(t,h)} 
| K(t,s) - K(t+h,s) |_{\cL(\X ,\Y )}^{r} \,\ud s\right)^{\frac{1}{r}}<\infty, \tag{$D_{r,-}'$}
\]
where
\begin{align*}
 I_m(s,h)&:=\{ t\in \R_-: 2^m h < s-t \leq 2^{m+1} h\}, \quad \textrm{and} \\
J_m(t,h)&:=\{ s\in \R_-: 2^m h < s-t \leq 2^{m+1} h\} \quad (m\in\N).
\end{align*}
Note that the condition $(D_{1,-})$ can be rewritten as  
\[
[K]_{D_{1,-}} = \sup_{s'<s<0}\int_{\{2(s-s')\leq s-t\}} | K(t,s) - K(t,s') |_{\cL(\X ,\Y )}\,\ud t <\infty .
\]

\begin{theorem}\label{thm.exsio} 
Let  $T$  be a singular integral operator associated with a kernel $K$ supported on $\{ (t,s)\in R_-\times\R_-: t<s\}$. Assume that $K$ 
satisfies the conditions $(D_{1,-})$ and $(D_{r, -}')$ for some 
$r\in (1,\infty)$. Then, for every $p\in (0,\infty )$ and  for every weight $w\in A^+_\infty(\R_-)$, there exists a constant $C=C(p, w, T, [K]_{D_{1,-}}, [K]_{D_{r, -}'})$ such that 
\begin{equation}\label{r}
\int_{\R_-} |Tf|_\Y^p w \;\ud t \leq C \int_{\R_-} \left(M^+_-(|f|_\X^{r'})\right)^{p/r'} w \;\ud t .
\end{equation} for every $f\in L^\infty_c(\R_-; X)$ with $M^+_-\left(|Tf|_\Y \right)\in L^p_w(\R_- )$.

Furthermore, if $\cF\subseteq A^+_p(\R_-)$ ($p\in (1,\infty)$) with $\sup_{w\in \cF}[w]_{ A^+_p(\R_-)}<\infty$, then the constants $C$ can be chosen such that 
\[
\sup_{w\in \cF} C(p, w, T, [K]_{D_{1,-}}, [K]_{D_{r, -}'})< \infty.
\] 
\end{theorem}

The proof is divided into two lemmas. As in the classical case, the main ingredient of the proof of Theorem \ref{thm.exsio} is a variant of the Fefferman-Stein inequality corresponding to the class $A_\infty^+(\R_-)$. To formulate it we start with some preliminaries.

Recall that the one-sided sharp maximal operator $M^{+,\sharp}$ corresponding to the  operator $M^+$ was introduced in Martin-Reyes \& de la Torre \cite{MRTo94} and is given by
\begin{align*}
M^{+,\sharp}f(t) & := \sup_{h>0} \frac{1}{h}\int_t^{t+h} \left( f(s) - \frac{1}{h}\int_{t+h}^{t+2h}f \,\ud\tau \right)^+ \!\ud s \text{ for every } f\in L_{loc}^1(\R). 
\end{align*}
Note that $M^{+,\sharp}f\leq 3 M^+f$ for every $f\in L_{loc}^1(\R)$. We set 
\[
M^{+,\sharp}_- f:=\chi_{\R_-}M^{+,\sharp}f\quad \textrm{ for every } f\in L_{loc}^1(\R_-).
\]
In \cite[Theorem 4]{MRTo94}, Martin-Reyes and de la Torre proved an analogue of the Fefferman-Stein inequality for the operators $M^+$ and $M^{+,\sharp}$, and for Sawyer's class $A_\infty^+(\R)$. The next result is a variant of \cite[Theorem 4]{MRTo94} for the operators $M^+_-$ and $M^{+,\sharp}_-$, and the class $A_\infty^+(\R_-)$.

\begin{lemma}\label{FSineq}
 Assume that $w\in A_\infty^+(\R_-)$, and let $f\geq 0$ be locally integrable such that $M^+_-f\in L^{p_0}_w(\R_- )$ for some $p_0\in(0,\infty)$. Then, for every $p\geq p_0$ there exists a constant 
 $C=C(p, w)$ such that 
\[
\int_{\R_-} (M^+_-f)^p w\, \ud t \leq C \int_{\R_-} (M^{+,\sharp}_- f)^p w\, \ud t.
\]
Furthermore, if $\cF\subseteq A^+_p(\R_-)$ ($1<p<\infty$) satisfies $\sup_{w\in \cF}[w]_{A^+_p(\R_-)}<\infty$, then the constants $C(p,w)$ can be chosen in such a way that
\begin{equation}\label{const}
 \sup_{w\in \cF} C(p, w)< \infty.
\end{equation} 
\end{lemma}

The proof of Lemma \ref{FSineq} can be obtained by an adaptation of the techniques developed in \cite{MRTo94}. 
The qualitative information on the constants involved in the various inequalities which leads to our second statement is not stated explicitly in \cite{MRTo94}. Since the second statement in Lemma \ref{FSineq} is crucial for the proof of Theorem \ref{Boyd th} below, for the convenience of the reader we sketch the proof and underline the steps which lead to (\ref{const}).

The proof is based on the following property of the weights in  $A^+_\infty(\R_-)$. See \cite[Theorem 1]{MRPiTo93} for the corresponding result for Sawyer's class. 

\begin{proposition}\label{lem.dubbling}
For every  $w\in A^+_\infty(\R_-)$ there exist constants $C$, $\delta >0$ such that 
\begin{equation}\label{eq.dubbling}
 \frac{w(S)}{w(a,c)}\leq C \left(\frac{|S|}{c-b}\right)^\delta
\end{equation}
for every $a<b<c\leq 0$, and every measurable set $S\subseteq (a,b)$.

Furthermore, if $\cF\subseteq A_\infty^+(\R_-)$ such that 
there exist constants $\delta$, $C$ for which condition (\ref{Ainfty-}) holds uniformly for $w\in \cF$, then (\ref{eq.dubbling}) holds uniformly in $\cF$, too.
\end{proposition} 

\begin{proof}
Consider first the following statements:
\begin{itemize}
 \item [(i)]  $w\in A^+_\infty(\R_-)$. 
 \item [(ii)] For every $\alpha \in(0,1)$ there exists $\beta >0$ such that 
 for every $a<b<c\leq 0$ and every measurable set $S\subseteq (b,c)$, 
 if $w(S)<\beta w(a,b)$, then $|S| <\alpha (c-a)$.
 \item [(iii)] For every $\alpha \in(0,1)$ there exists $\beta >0$ such that if 
 $\lambda >0$ and $a<b<0$ satisfy 
\[
\lambda = \frac{1}{b-a}\int_a^bw\;\ud t\leq \frac{1}{s-a}\int_a^sw\;\ud t\quad\quad \text{for every } s\in (a,b),
\]
 then $|\{ t\in (a,b): w(t)>\beta \lambda \}|>\alpha(b-a)$.
 \item [(iv)] There exist constants $\delta$, $C>0$ such that for every $a<b<0$ we have the following variant of the reverse H\"older inequality
\[
\frac{1}{b-a} \int_a^b w^{1+\delta}\; \ud t \leq C \, \left(M^-(w \chi_{(a, b)})(b) \right)^{1+\delta}.
\]
\end{itemize} 
Then we have (i)$\Rightarrow$(ii)$\Rightarrow$(iii)$\Rightarrow$(iv). The implication (i)$\Rightarrow$(ii) follows immediately from the definition of the $A^+_\infty (\R_-)$ condition. 
For the proof of the implications (ii)$\Rightarrow$(iii) and (iii)$\Rightarrow$(iv) one can 
adapt the proof of \cite[Theorem 1, $(c)\Rightarrow (d)$ and $(e)\Rightarrow (f)$]{MRPiTo93}, which follows in principle the arguments used in \cite{MR93} and is reproduced here in the proofs of Lemmas \ref{lem.weak-estim} and \ref{lem.openness} above; see, in particular, the proof of \eqref{weakrH}.

Then, (\ref{eq.dubbling}) follows in a straightforward way from (iv) and the fact that $M^-$ is of weak type $(1,1)$; see, for example, \cite[Theorem 1]{MR93}. More precisely, the reverse H\"older inequality yields the 
existence of constants $\delta$, $C>0$ such that for every $a<b<c\leq 0$ 
\begin{equation*} 
 \frac{1}{w(a,c)}\int_a^b w^{1+\delta}\, \ud t < C \left(M^-(w\chi_{(a,c)})(s)\right)^\delta \quad \text{for every } s\in (b,c).
\end{equation*}
Assume that $c<0$.  Set $\lambda:=\left(\frac{1}{C w(a,c)}\int_a^b w^{1+\delta}\;\ud t\right)^{1/\delta}$, where $C$ is as in the preceding inequality. Then $\lambda >0$ and  
\[
 (b,c) \subseteq \{ s\in \R_-: M^-(w\chi_{(a,c)})(s) >\lambda \} .
\]
This and the weak $(1,1)$ inequality for $M^-$,
\begin{equation*}
|\{ s\in \R_-: M^-(w\chi_{(a,c)})(s) >\lambda \}| \leq \frac{C_0}{\lambda} \int_a^c w\,\ud t \quad\text{for every } \lla>0 ,
\end{equation*}
imply
\[
c-b\leq C_0 C^{1/\delta} \left(\frac{\int_a^c w\, \ud t}{\int_a^b w^{1+\delta}\,\ud t}\right)^{1/\delta} \int_a^c w\,\ud t .
\]
By H\"older's inequality, for every measurable set $S\subseteq (a,b)$ with $|S|>0$,
\begin{align*}
\left( \int_S w\, \ud t\right)^{1+1/\delta}\, |S|^{-1} & \leq \left(\int_S w^{1+\delta}\, \ud t\right)^{1/\delta} \\
& \leq \left(\int_a^b w^{1+\delta}\, \ud t\right)^{1/\delta} .
\end{align*}
This inequality combined with the preceding inequality yields 
\[
\frac{w(S)}{w(a,c)} \leq C_0 C^{1/\delta} \left(\frac{|S|}{c-b}\right)^{\frac{\delta}{1+\delta}} \text{ for every measurable } S\subseteq (a,b) .
\]
Since $\int_b^0 w\,\ud t\in (0,\infty]$, a limiting argument allows one to drop the assumption that $c<0$. Finally, note that the construction of the constants $\alpha$, $\beta$, $\delta$, $C$ in the statements (ii), (iii) and (iv) above follows that of the corresponding constants in Lemmas \ref{lem.weak-estim} and \ref{lem.openness}; in particular, these constants may be chosen uniformly in $w\in\cF$, if $\cF\subseteq A_\infty^+ (\R_- )$ is as in the second part of the statement. The inequality above then yields the claim. 
\end{proof} 

\begin{proof}[Proof of Lemma \ref{FSineq}]
By using Lemma \ref{lem.dubbling} and following the lines of the proof of 
\cite[Theorem 4]{MRTo94} we get the following variant of the good $\lla$ inequality: 
\[
w\left(\{ t\leq a: M^+_-f(t) > 2 \lla, \; M^{+,\sharp}_-f(t)\leq \gamma \lla  \}\right)\leq 4^\delta C  \gamma^\delta w\left(\{ t\leq a: M^+_-f(t) >\lla  \}\right)
\]
for every $\gamma \in (0,1)$, $\lla>0$, and $a<0$, where $\delta$ and $C$ are the constants from \eqref{eq.dubbling}. 
Then, a standard argument gives 
\begin{align*}
 \int_0^N p\lla^{p-1} w( \{ t\leq a: & M^+_-f(t) >\lla  \} )\, \ud \lla\\
&\leq {\frac{2^{p+1}}{\gamma^p}} \int_0^{\gamma N/2} p\lla^{p-1} w(\{ t\in \R_-: M^{+,\sharp}_-f(t) >\lla  \})\, \ud \lla 
\end{align*} 
for every $a<0$ and $N>0$, where $\gamma:= 4^{-1}(2^{p+1}C)^{-1/\delta}$.
A limiting argument yields the first statement of Lemma \ref{FSineq}. 

For the second one, note that the condition (\ref{Ainfty-}) holds uniformly in $\cF$. 
Indeed,  one can take $C:=\sup_{w\in \cF}\|M^+_-\|_{p,w}<\infty$
and $\delta := 1/p$; see Corollary \ref{unif bound}.
\end{proof}

The second lemma is a counterpart of \cite[Theorems 1.2 and 1.3, Part III]{RuRuTo86}. 

\begin{lemma} \label{lem.point}
With the same assumptions on $T$  as in Theorem \ref{thm.exsio} there exists a constant $C=C(T, [K]_{D_{1,-}}, [K]_{D_{r,-}'})$ such that 
 \begin{equation}\label{point.est}
  M^{+,\sharp}_-(|Tf|_\Y )\leq C \left(M^+_-(|f|_\X^{r'})\right)^{1/r'} \text{ for every } f\in L_c^\infty(\R_-; \X) .
 \end{equation}
\end{lemma}

The proof is standard and reproduces essentially ideas which have been presented in \cite{BeCaPa62} and \cite{RuRuTo86}. We give the details for the convenience of the reader; cf. also the approach based on Kolmogorov's inequality in \cite[Theorem 3]{LoRiTo05} adapted from \cite{AlPe94}.

\begin{proof} 
According to the definition, $T$ is bounded from $L^p(\R_+; \X)$ into $L^p (\R_+;\Y)$ for some $p\in (1,\infty)$. Similarly as in \cite{BeCaPa62}, one first shows $T$ is of weak type $(1,1)$. This part of the proof uses the assumption that the kernel satisfies the $(D_{1,-})$ condition, and will not be repeated here. By the Marcinkiewicz interpolation theorem, $T$ extends to a bounded linear operator from $L^q(\R_+;\X)$ into $L^q(\R_+;\Y)$ for every $q\in(1,p]$. 

To prove \eqref{point.est} we proceed as follows. 
Fix $f\in L_c^\infty(\R_-; \X)$ and $t<0$.
One can easily check that 
\[
M^{+,\sharp}_- (|Tf|_\Y) (t) = M^{+,\sharp} (|Tf|_\Y)(t)\leq 4 \, \sup_{h>0}\inf_{y\in \Y}\frac{1}{h} 
\int_{t}^{t+h} |Tf(t) - y|_\Y \, \ud t. 
\]
Therefore, it is sufficient to show that there exists a constant $C$ such that for every $h>0$ and some $y\in\Y$ we have 
\begin{equation}\label{lem3inq4}
 \frac{1}{h} 
\int_{t}^{t+h} |Tf(t) - y|_\Y \, \ud t\leq C \, \left(M^+_-(|f|_\X^{r'})(t)\right)^{1/r'}.
\end{equation}
For $0<h <-t/2$, set $f_1=f\chi_{(t,t +2h)}$, $f_2:=f\chi_{[t+2h, 0]}$, and $y:=Tf_2(t)$.
Note that, for every $\tau \in(t, t+h)$,
\begin{align*}
& T(f\chi_{(-\infty, t]})(\tau) = 0 \; \text{ and } \;    
 Tf_2(\tau) - y = \int_{t+2h}^0 \left( K(t',s) - K(t,s) f(s)  \right)\ud s .
\end{align*} 
Consequently,  we obtain
\[
 \frac{1}{h} 
\int_{t}^{t+h} |Tf(\tau)- y|_\Y \,\ud \tau\\
\leq \frac{1}{h}\int_{t}^{t+h} |Tf_1(\tau)|_\Y \,\ud \tau 
+ \sup_{\tau\in (t, t+h)}  | Tf_2(\tau) - y|_\Y
\]
Since $T$ is bounded from $L^q(\R_-; \X)$ into $L^q(\R_-;\Y)$ for some $q<r'$, by H\"older's inequality, we get
\[
\frac{1}{h}\int_{t}^{t+h} |Tf_1(\tau)|_\Y \ud \tau \leq |T|_{\cL(L^q)} \left(M^+_-(|f|_\X^{r'})(t)\right)^{1/r'}. 
\]
For the second summand, since $[t+2h,0] \subseteq \bigcup_{m\in \N} J_m(t,\tau - t)$, by the $(D_{r,-}')$ condition, we easily obtain that
\begin{align*}
& | Tf_2(\tau) - y|_\Y \leq \int_{t +2h}^0 | K(\tau ,s) - K(t,s)|_{\cL(\X ,\Y)}\, |f(s)|_\X \;\ud s\\
& \leq  \sum_{k\in \N} \left( \int_{J_k(t,\tau - t)} | K(\tau,s) - K(t,s)|_{\cL(\X,\Y)}^r \; \ud s\right)^{1/r}\left( \int_{J_k(t,\tau -t )}  |f(s)|_\X^{r'} \,\ud s\right)^{1/r'} \\
& \leq 2^{1/r'}[K]_{D_{r,-}'} \left(M^+_-(|f|_\X^{r'})(t)\right)^{1/r'}
\end{align*} for every $\tau\in (t, t +h)$.

For $h\geq -t/2$, applying the boundedness of $T$ on $L^q(\R_-;\X)$ for some $q<r'$, 
one can easily get \eqref{lem3inq4} with $y=0$.
Thus, this completes the proof.
\end{proof}

\begin{proof}[Proof of Theorem \ref{thm.exsio}]
Combining Lemmas \ref{FSineq} and \ref{lem.point} with Lebesgue's differentiation theorem we can now proceed as follows:
\begin{align*}
\int_{\R_-} |Tf|_\Y^p w\;\ud t &\leq \int_{\R_-} \left(M^{+}_-(|Tf|_\Y)\right)^{p} w\,\ud t \\
& \leq C \, \int_{\R_-} \left(M^{+,\sharp}_-(|Tf|_\Y)\right)^{p} w\,\ud t\\
& \leq C \, \int_{\R_-} \left(M^{+}_-(|f|_\X^{r'})\right)^{p/r'} w\,\ud t.
\end{align*}
The second statement follows simply from the corresponding one of Lemma \ref{FSineq}.
\end{proof}

Note that a simple change of variables gives the following equivalent formulation of Theorem \ref{thm.exsio}, which is a starting point for our further consideration in the next sections.  
The corresponding symmetric conditions to $(D_{r,-})$ and $(D_{r,-}')$ $(r\in [1,\infty])$ one can explicitly express as follows: 
\[
[K]_{D_{r,+}}:=\sup_{s>0, h>0}h^{\frac{1}{r'}}
\sum_{m=1}^\infty 2^{\frac{m}{r'}} \left(\int_{I^+_m(s,h)} 
| K(t,s) - K(t,s+h) |_{\cL(\X ,\Y )}^{r} \ud t\right)^{\frac{1}{r}}<\infty, \tag{$D_{r,+}$}
\]

\[
[K]_{D_{r,+}'}:=\sup_{t>0, h>0}h^{\frac{1}{r'}}
\sum_{m=1}^\infty 2^{\frac{m}{r'}} \left(\int_{J^+_m(t,h)} 
| K(t,s) - K(t-h,s) |_{\cL(\X ,\Y )}^{r} \ud s\right)^{\frac{1}{r}}<\infty, \tag{$D_{r,+}'$}
\]
where
\begin{align*}
 I^+_m(s,h) &:=\{ t\in \R_+: 2^m h < t-s \leq 2^{m+1} h\}, \quad \textrm{and} \\
J^+_m(t,h) &:=\{ s\in \R_+: 2^m h < t-s \leq 2^{m+1} h\} \quad (m\in\N). 
\end{align*}

\begin{theorem} \label{exsio+}
 Let $T$ be a singular integral operator associated with kernel $K$ 
 supported in $\{(t,s)\in\R_+\times\R_+ : t>s \}$ and satisfying the conditions $(D_{1,+})$ and $(D_{r,+}')$ for some $1 < r < \infty$. Then, for every $0<p<\infty$ and for every weight $w\in A^-_\infty(\R_+)$, there exists a constant $C=C(p, w, T, [K]_{D_{1,+}}, [K]_{D_{r, +}'})$ such that 
\begin{equation}\label{r+}
\int_{\R_+} |Tf|_\Y^p w\; \ud t \leq C \int_{\R_+} \left(M^-_+(|f|_\X^{r'})\right)^{p/r'} w \;\ud t
\end{equation}
for every $f\in L^\infty_c(\R_+; \X)$ with $M^-_+\left( |Tf|_\Y\right)\in L^p_w(\R_+ )$. 

Furthermore, if $\cF\subseteq A^-_p(\R_+)$ $(1<p<\infty)$ with $\sup_{w\in \cF}[w]_{ A^-_p(\R_+)}<\infty$, then the constants $C$ can be chosen such that
\[
\sup_{w\in \cF} C(p, w, T, [K]_{D_{1,+}}, [K]_{D_{r, +}'})< \infty.
\] 
\end{theorem}

\section{Rearrangement invariant Banach function spaces} \label{sec.r.i.extrap}

In this section we apply techniques from interpolation and extrapolation theory to establish further boundedness properties of the operators discussed in Sections \ref{sec.Muckenhoupt} and
\ref{sec.Coifman}.

The first result provides a counterpart of the Lorentz-Shimogaki theorem for the operators $M^\pm_-$ and $M^\pm_+$; see Proposition \ref{LSanalogue}. Next, applying techniques of Rubio de Francia's extrapolation theory, we show that for singular integral operators satisfying the assumptions of Theorem \ref{exsio+} for every $r\in(1,\infty)$ an analogue of the classical Boyd theorem holds; see Theorem \ref{Boyd th}. We restrict our considerations to the case of the positive half-line $\R_+$, but corresponding results hold in the case of $\R_-$. We start with some preparation. \\

Throughout, let $\E$ be a rearrangement invariant Banach function space over $(\R_+, \ud t)$.
Denote by $\cM^+(\R_+)$ the set of all nonnegative measurable functions on $\R_+$. Let $w$ be a weight on $\R_+$ which is locally integrable on $(0, \infty)$.
Define $\rho_w:\cM^+(\R_+)\rightarrow [0,\infty]$ by $\rho_w(f) := \|f^*_w\|_{\E}$, where $f^*_w$ denotes the decreasing rearrangement of $f$ with respect to $w\,\ud t$. By \cite[Theorem 4.9, p.\,61]{BeSh88}, $\rho_w$ is a rearrangement invariant Banach function norm with respect to $(\R_+, w\,\ud t)$. Write $\E_w$ for the rearrangement invariant Banach function space corresponding to $\rho_w$, and also $\|\cdot\|_{\E_w}$ for the norm of $\E_w$.
Note that $L^p$-spaces with respect to the weighted Lebesgue measure $w\; \ud t$ (which have been denoted by $L^p_w$ up to now) and the $L^p_w$-spaces as defined in this paragraph coincide so that there is no danger of ambiguity in our notation.

Following \cite{LiTz79}, we define 
the \it {lower} \rm and \it {upper Boyd indices} \rm respectively by 
\begin{align*}
p_\E & = \lim_{t\to \infty} \frac{\log t}{\log h_\E(t)} = \sup_{1<t<\infty} \frac{\log t}{\log h_\E(t)} \quad \text{ and} \\
q_\E & = \lim_{t\to 0+} \frac{\log t}{\log h_\E(t)} = \inf_{0<t<1} \frac{\log t}{\log h_\E(t)},
\end{align*}
where $h_{\E}(t)=\|D_t\|_{\cL (\E)}$ and 
$D_t: \E\rightarrow \E$ $(t>0)$ is the \it {dilation operator} \rm defined by 
\[
D_tf(s)=f(s/t), \qquad (0<t<\infty, \, f\in \E ).
\]
One always has $1\le p_\E\le q_\E\le\infty$, see for example \cite[Proposition 5.13, p.\,149]{BeSh88}, where the Boyd indices are defined as the reciprocals with respect to our definitions. In particular, we have $p_\E=q_\E=p$ for $\E=L^{p,q}$ ($1<p<\infty$, $1\leq q\leq \infty$).

\begin{proposition}\label{LSanalogue}
 Let $\E$ be any rearrangement invariant Banach function space over $(\R_+, \ud t)$ with 
Boyd indices $p_\E, q_\E\in (1, \infty)$. Then,  for every weight $w\in A^\pm_{p_\E}(\R_+)$ the operator $M^\pm_+$ is bounded on $\E_w$. 
\end{proposition}

\begin{proof} 
 By Theorem \ref{thm.restSawyer}, the operator $M^\pm_+$ is bounded on $L^p_w(\R_+ )$ for every $p\in (1,\infty)$ and every $w\in A_p^\pm(\R_+)$.
 Fix $w \in A_{p_\E}^\pm(\R_+)$, and $q \in (q_\E,\infty)$. By the openness property of the $A^\pm_p(\R_+)$ classes, Lemma \ref{lem.openness}, there exists $r\in (1,p_\E)$ such that $w\in A_r^\pm(\R_+)$. In particular, by \cite[Theorem 4.11, p.\,223]{BeSh88} (see also \cite[Theorem 8]{Ca66}), $M^\pm_+$ is of joint weak type $(r,r;q,q)$ with respect to $(\R_+, w\,\ud t)$. More precisely, according to \cite[Definition 5.4, p.\,143]{BeSh88}, for every $f\in L^{r,1}_w+L^{q, 1}_w(\R_+ )$ and every $t>0$
\[ 
(M^\pm_+f)^*_w(t) \leq C S_\sigma (f^*_w)(t) ,
\]
 where $S_\sigma$ stands for the corresponding Calder\'on operator with $\sigma:=(r^{-1},r^{-1},q^{-1},q^{-1})$.
 By Boyd's theorem \cite[Theorem 5.16, p.\,153]{BeSh88}, $S_\sigma$ is bounded on $\E$. Therefore, we obtain
\[ 
\|M^\pm_+f\|_{\E_w} =\|(M^\pm_+f)^*_w\|_{\E}\leq C \|S_\sigma f^*_w\|_{\E}\leq 
 C\|S_\sigma\|_{\cL(\E)}  \| f\|_{\E_w}
\]
 for every $f\in L^{r,1}_w+L^{q, 1}_w(\R_+ ) $. Since $\E_w\subseteq L^{r,1}_w +L^{q, 1}_w(\R_+ )$, see Lemma \ref{ing1} below, the proof is complete. 
\end{proof}

Recall that if $\E$ is a rearrangement invariant Banach function space over $(\R_+, \ud t)$ and $1\leq p<p_\E$, $q_\E<q\leq \infty$, then 
\[
L^p\cap L^q(\R_+) \subseteq \E \subseteq L^p+L^q(\R_+) ; 
\]
see, for example, \cite[Proposition 2.b.3]{LiTz79}. For any weight $w$, the spaces $L^p_w\cap L^q_w(\R_+ )$ and $L^p_w+L^q_w(\R_+ )$ are endowed with the norms
\begin{align*}
 L^p_w\cap L^q_w(\R_+ )\ni\; & f\mapsto  \max(\|f\|_{L^p_w}, \|f\|_{L^q_w}) \quad  \textrm{ and}\\ 
 L^p_w + L^q_w(\R_+ )\ni\;  & f \mapsto \inf\left\{ \|g\|_{L^p_w} + \|h\|_{L^q_w}: g\in L^p_w, h\in L^q_w, g+h=f \right\},
\end{align*} respectively.
Recall also that, by Lemma \ref{lem.openness}, 
\[
q_w:= \inf\{ q\in [1,p]:  w\in A_q^-(\R_+) \}<p\quad  \textrm{ for every } w\in A_p^-(\R_+)\;  (p>1).
\]

\begin{lemma} \label{ing1} 
Let $\E$ be a rearrangement invariant Banach function space over $(\R_+, \ud t)$ with Boyd indices $p_\E, q_\E \in(1,\infty)$. Then, the following statements hold. 
\begin{itemize}
 \item [(i)] Let $w$ be any weight on $\R_+$ which is locally integrable on $(0,\infty)$. Then, for every $ p\in [1,p_\E)$ and $q\in (q_\E , \infty]$ we have: 
 \begin{equation}\label{embed}
  L^p_w \cap L^q_w(\R_+ ) \subseteq \E_w \subseteq L^p_w+L^q_w (\R_+ ) ,
 \end{equation}
 and the embeddings are continuous.
 \item [(ii)] For every $w\in A^-_{p_\E}(\R_+)$ and for every $r\in [1,p_\E/q_w)$, we have 
\begin{equation*}
 \E_w  \subseteq L^r_{loc}(\R_+)
\end{equation*}
and the embedding is continuous. 
\end{itemize}
\end{lemma}

\begin{proof} 
(i) We follow the idea of the proof of \cite[Proposition 2.b.3]{LiTz79}.
 Let $f\in L^p_w \cap L^q_w(\R_+ )$ be a simple function with $\|f\|_{L^p_w},
 \|f\|_{L^q_w}\leq 1$. 
Let $g$ be a simple function such that $g(t)=\sum_{k\in I} 2^k \chi_{A_k}$, and 
$f(t)/2\leq g(t)\leq f(t)$ ($t>0$), where $I$ is a subset of $\mathbb{Z}$. In particular,  since $w(A_k)<\infty$, $\chi_{A_k}\in \E_w$ ($k\in I$). 
Hence, we obtain that 
\begin{align*}
 \|g\|_{\E_w} & \leq \sum_{k\in I} 2^{k}\left \|(\chi_{A_k})^*_w \right\|_{ \E} \\
 & = \sum_{k\in I} 2^{k}\left \|\chi_{[0,w(A_k)]} \right\|_{ \E} \\
 & = \sum_{k\in I} 2^{k}\left \|D_{w(A_k)}\chi_{[0,1]} \right\|_{ \E}. 
\end{align*}
Fix $q\in (q_\E, \infty)$. By definition of the Boyd indices, for every $r\in (1,p_\E )$ and $s\in (q_\E ,\infty )$ there exists a constant $C$ such that 
\begin{equation}\label{dilation}
\begin{split}
& \|D_u\|_{\cL( \E)}\leq C \, u^{1/r} \text{ for every } u\in [1,\infty ), \textrm{ and }  \\
& \|D_u\|_{\cL(\E)} \leq  C \, u^{1/s}\text{ for every } u\in (0, 1).
\end{split}
\end{equation}
Note that $w(A_k)\leq \min(2^{-kp}, 2^{-kq})$ ($k\in I$). Set 
\[
I_1:=\{ k\in I: w(A_k)\geq 1\} \quad \textrm{and}\quad I_2:=\{ k\in I: w(A_k)\leq 1\}. 
\] 
Fix $r$ and $s_i$ $(i=1,2)$ such that $p<r<p_\E$ and $q_\E< s_1 < q< s_2$. Then, by 
(\ref{dilation}), there exists a constant $C$ such that 
\begin{align*}
\|D_{w(A_k)}\|_{\cL( \E)} & \leq C \, w(A_k)^{1/r} \; \quad (k\in I_1), \text{ and} \\ 
\|D_{w(A_k)}\|_{\cL( \E)} & \leq  C \, w(A_k)^{1/s_i}\; \quad (k\in I_2, \, i=1,2).
\end{align*}
Since $I_1\subseteq \ZZ_-$, 
\[
\sum_{k\in I_1} 2^k \|D_{w(A_k)}\|_{\cL( \E)} \leq C \sum_{k\in I_1} 2^k w(A_k)^{1/r}\leq C \sum_{k\in I_1} 2^k 2^{-k\frac{p}{r}}<\infty.
\]
Furthermore, 
\begin{align*}
 \sum_{k\in I_2} 2^k \|D_{w(A_k)}\|_{\cL(\E)} &\leq C \, \sum_{k\in I_2, k\geq 0} 2^k w(A_k)^{1/s_1} + C \sum_{k\in I_2, k< 0} 2^k w(A_k)^{1/s_2} \\
 &\leq C \, \sum_{k\in I_2, k\geq 0} 2^k 2^{-k\frac{q}{s_1}} + C \sum_{k\in I_2, k< 0} 2^k 2^{-k\frac{q}{s_2}}<\infty.
\end{align*}

To show the second inclusion in (\ref{embed})  recall first that by applying Luxemburg's representation theorem, \cite[Theorem 4.10, p.\,62]{BeSh88} one can show that 
\[
\frac{1}{p_\E} + \frac{1}{q_{\E'}} =1 \quad \quad \frac{1}{q_\E} + \frac{1}{p_{\E'}}=1,
\]
where $\E'$ stands for the associated space of $\E$; see \cite[Definition 2.3, p.\,9]{BeSh88}.
Set $\bX:=\E'$. Since $q'<p_\bX$, $q_\bX<p'$, from what has already been proved we get that $L^{q'}_w\cap L^{p'}_w(\R_+ ) \subseteq \bX_w$ continuously. Hence, by the \emph{duality} argument we get that $(\bX_w)'$ is continuously embedded into $L^{p}_w + L^{q}_w(\R_+ )$. 

Note that $(\bX')_w\subseteq (\bX_w)'$. Indeed, since $(\R_+,w\,\ud t)$ is resonant (see \cite[Theorem 2.7, p.\,51]{BeSh88}), by 
\cite[Proposition 4.2]{BeSh88} and Landau's resonance theorem, \cite[Lemma 2.6, p.\,10]{BeSh88}, a function $f\in \cM(\R_+)$ belongs to $(\bX_w)'$ if and only if $\int_0^\infty f^*_w g^*_w \ud t<\infty$ for every $g\in \E_w$. On the other hand, by Luxemburg's representation theorem and Landau's resonance theorem, a function $f\in \cM(\R_+)$ belongs to $(\bX')_w$ if and only if $\int_0^\infty h f_w^* \ud t <\infty$ for every $h\in \E$. Therefore, the claimed embedding holds, and its continuity follows from standard arguments.

Finally, note that by the Lorentz-Luxemburg theorem, \cite[Theorem 2.7, p.\,10]{BeSh88}, $\bX'=(\E')'=\E$. This completes the proof of (i).

(ii) Fix $w\in A^-_{p_\E}(\R_+)$. It is sufficient to prove the claim for $r<p_\E/q_w$ such that $r=p/s$ for $q_w < s < p < p_\E$. Then, by (i), we conclude that $\E_w\subseteq L^p_{w,loc}(\R_+ )$, and by the one-sided Muckenhoupt condition $w^{1-s'}\in L^1_{loc}(\R_+)$. Therefore, for $f\in \E_w$ and $a>0$, H\"older's inequality yields
\begin{align*}
 \left( \int_0^a |f|^p w\; \ud t \right)^{1/s} \left( \int_0^a w^{1-s'}\; \ud t \right)^{1/s'} & = \left( \int_0^a |f|^{rs} w\; \ud t \right)^{1/s} \left( \int_0^a w^{-s'/s}\; \ud t \right)^{1/s'}\\
 & \geq \int_0^a |f|^r\; \ud t.
\end{align*}
\end{proof}

We are now in a position to apply Theorem \ref{LSanalogue} to provide weighted rearrangement invariant inequalities for singular integral operators on the half-line. 

First, as in Curbera, Garc{\'{\i}}a-Cuerva, Martell and P{\'e}rez \cite{CGMP06}, we define the vector-valued version $\E_w(\R_+;\X)$ of the rearrangement invariant Banach function space $\E_w(\R_+)$ in the following way:
\[
\E_w (\R_+;\X ) :=\left\{ f: \R_+\rightarrow \X \;\;{\rm{measurable }}:  |f|_\X \in \E_w (\R_+ ) \right\} ,
\]
and its norm is $\|f\|_{\E_w(\R_+;\X)}: =\| |f|_\X \|_{\E_w}$.

\begin{theorem}\label{Boyd th} 
Let $T$ be a singular integral operator  associated with kernel $K$ supported in $\{(t,s)\in\R_+\times\R_+: t>s\}$ and  satisfying the conditions $(D_{1,+})$ and $(D'_{r,+})$ for every $1 < r < \infty$. 

Then, for every rearrangement invariant Banach function space $\E$ with Boyd indices $p_{\E}$, $q_{\E}\in (1, \infty )$, and for every weight $w\in A^-_{p_\E}(\R_+)$,  $T$ extrapolates to a bounded linear operator from $\E_w(\R_+;\X)$ into $\E_w(\R_+;\Y)$.
\end{theorem}

The proof of Theorem \ref{Boyd th} follows in principle the idea of the proof of \cite[Theorem 7]{ChKr14}. We provide only main supplementary observations which should be made.

\begin{proof}
We first show that $T$ extends to a bounded operator from $L^p_w(\R_+ ; \X)$ into $L^p_w(\R_+;\Y)$ for every $p\in (1,\infty)$ and every weight $w \in A^-_p(\R_+)$.

Fix $p\in (1,\infty)$ and $w \in A^-_p(\R_+)$. Set $w_k:=\inf (w, k)$ ($k\geq 1$). Then,  one can easily show that $w_k\in A^-_p(\R_+)$ and $[w_k]_{A^-_p(\R_+)}\leq 2^p (1+ [w]_{A^-_p(\R_+)})$ for every $k\geq 1$. Moreover, by Lemma \ref{lem.openness}, there exists  $q\in(1,p)$ such that $w\in A^-_{q}(\R_+)$.

Combining Theorem \ref{exsio+} (for $w=\chi_{\R_+}$) with Theorem \ref{thm.restSawyer}, note that $T$ is bounded on $L^p(\R_+; X)$. In particular, $|Tf|_\Y \in L^p_{w_k}$, and consequently, by Theorem \ref{thm.restSawyer}, $M^-_+\left( |Tf|_\Y\right) \in  L^p_{w_k}$ for every $k\geq 1$ and every $f\in L^\infty_c (\R_+; \X)$. 
Therefore, we can again apply Theorem \ref{exsio+} for $w_k$ and $r$ such that $p/r' = q$, to obtain, for every $f\in L_c^\infty (\R_+; \X)$, 
\begin{align*}
\int_{\R_+} |Tf|_\Y^p w_k \; \ud t &  \leq 
 C(p, w_k, T, [K]_{D_{1,+}}, [K]_{D_{r,+}'}) \,  
\int_{\R_+} \left(M^-_+ (|f|_\X^{r'})\right)^{p/r'} w_k \; \ud t\\
&\leq   \, 
C(p, w_k, T, [K]_{D_{1,+}}, [K]_{D_{r,+}'}) \, \|M^-_+\|^{q}_{q,w_k}
\int_{\R_+}  |f|_\X^{p} w_k \; \ud t.
\end{align*} 
By Corollary \ref{unif bound} and the second statement of Theorem \ref{exsio+} we get 
\[ 
C(p,w,T, [K]_{D_{1,+}}, [K]_{D_{r,+}'}):= \sup_{k\geq 1}C(p, w_k, T, [K]_{D_{1,+}}, [K]_{D_{r,+}'}) \, \|M^-_+\|^{q}_{q,w_k}<\infty.
\]
Thus, letting $k\rightarrow \infty$ in the above inequalities, we thus obtain, for every $f\in L_c^\infty (\R_+; \X)$,
\begin{equation}\label{Coifmantype}
 \int_{\R_+} |Tf|_\Y^p w \; \ud t \leq C(p,w,T, [K]_{D_{1,+}}, [K]_{D_{r,+}'}) \,
\int_{\R_+}  |f|_\X^{p} w \; \ud t.
\end{equation}
 Since the space of all functions in $L^\infty_c ((0,\infty ); \X)$ is dense in $L^p_w(\R_+;\X)$, $T$ extends to a bounded operator from $L^p_w(\R_+;\X)$ into $L^p_w(\R_+;\Y)$ as we claimed. Moreover, if $\cF\subseteq A_p^-(\R_+)$ with $\sup_{w\in \cF}[w]_{A_p^-(\R_+)}<\infty$ then $\sup_{w\in \cF}C(p,w,T , [K]_{D_{1,+}}, [K]_{D_{r,+}'}) <\infty$; note that the $q$, which has been chosen above and which determines $r$, can be chosen uniformly in $w\in\cF$, and see Corollary \ref{unif bound}.
 
We are now in a position to adapt Rubio de Francia's iteration algorithm to the class $A_p^-(\R_+)$ and the operators $M^\pm_+$. See for example the proofs of \cite[Theorem 4.10 and Section 3.6]{CrMaPe11}.

Fix $\E$ and $w\in A^-_{p_\E}(\R_+)$ as in the assumptions. Let $\E_{w}'$ be the associate space of $\E_w$. Let  $\cR= \cR_w :\E_{w} \rightarrow \E_{w}$ and $\cR'= \cR'_w :\E_{w}' \rightarrow \E_{w}'$ be defined by
\begin{align*}
\cR h (t) & := \sum_{k=0}^{\infty} \frac{(M^-_+)^k h(t)}{2^k \|M^-_+\|^k_{\E,w}}, 
\text{ for every } 0\leq h\in \E_{w}, \text{ and} \\
\cR' h (t) & := \sum_{k=0}^{\infty} \frac{S^k h(t)}{2^k \|S\|^k_{\E',w}}, 
\text{ for every } 0\leq h\in \E'_{w},
\end{align*}
where $Sh:= M^+_+(h w)/w$ for $h\in \E'_{w}$. 
By Proposition \ref{LSanalogue}, the operators $\cR$ and $\cR'$ are well-defined. Indeed, for $\cR'$, note that the operator $S$ is bounded on $L^q_w$ for all $q\in (q_w, \infty)$ with $q_w <p_\E$; see Lemma \ref{lem.openness}. Therefore, a similar argument to that used in the proof of Proposition \ref{LSanalogue} yields the boundedness of $S$ on $\E_w$.

Moreover, the following statements are 
easily verified:
\begin{itemize}
\item [(i)]  For every positive $h\in\E_w$ one has 
\begin{align*}
& |h| \leq \cR h , \\
& \| \cR h\|_{\E_{w}}\leq 2 \|h\|_{\E_{w}} , \text{ and } \\
& \cR h\in A^+_1(\R_+) \text{ with }  [\cR h]_{A^+_1(\R_+)}\leq 2\|M^+_+\|_{\E,w} . 
\end{align*}
\item [(ii)]  For every positive $h\in\E_w'$ one has 
\begin{align*}
& h\leq \cR' h , \\
& \| \cR' h\|_{\E_{w}'}\leq 2 \|h\|_{\E_{w}'} ,  \text{ and }\\  
& (\cR' h) w\in A^-_1(\R_+) \text{ with } [(\cR' h) w]_{A^-_1(\R_+)}\leq 2\|S\|_{\E',{w}} .
\end{align*}
\end{itemize} 

Now fix $p\in (1,\infty )$. 
Note that $|f|_X\in L^p_{w_{f,h}}$ for every $f\in \E_w(X)$ and every positive
 $h\in \E'_w$, where
$w_{f,h}:=(\cR |f|_\X)^{1-p}(\cR' h) w$. By H\"older's inequality and the properties (i) and (ii) above, we obtain that  $w_{f,h} \in A^-_p(\R_+)$ and 
\begin{equation*}
[w_{f,h}]_{A^-_p(\R_+)}\leq [(\cR |f|_\X)]^{p-1}_{A^+_1(\R_+)}[(\cR' h)w]_{A^-_1(\R_+)}
\leq 2^{p}\|M^+_+\|^{p-1}_{\E,w} \|S\|_{\E',{w}}.
\end{equation*} 
Furthermore, by Corollary \ref{unif bound}, there exists $q<p$ and such that 
\begin{equation*}
 \sup\{ [w_{g,h}]_{A^-_{q}(\R_+)}: f\in \E_w (\X) \text{ and } 0\leq h\in \E'_w\}<\infty.
\end{equation*}
Thus, Theorem \ref{exsio+} (for $r$ such that $p/r'= q$) shows that there exists a constant $C\geq 0$ such that:
\begin{equation}\label{extrapineq}
\int_{\R_+} |Tf|_\Y^p w_{f,h} \; \ud t \leq C \int_{\R_+}  |f|_\X^{p} w_{f,h} \; \ud t 
\end{equation} 
for every $f\in \E_w(\R_+;\X)$ and every positive $h\in \E'_w$, where the constant $C$ is independent on $f$ and $h$. Now, using \eqref{extrapineq}, we can simply follow the corresponding idea in the proof of \cite[Theorem 4.10]{CrMaPe11} to obtain
\[
\int_{\R_+} |Tf|_\Y h w \; \ud t \leq 4C \,  \|f\|_{\E_{w}(\R_+;\X)} \, \|h\|_{\E'_{w}}
\]
for every $f\in \E_w(\R_+;\X)$ and every positive $h\in \E'_w$. 
Recall that, by the Lorentz-Luxemburg theorem, $\E_w = (\E_w')'$; see  \cite[Theorem 2.7,  p.\,10]{BeSh88}. Therefore, $T$ maps $\E_w(\R_+;\X)$ into $\E_w(\R_+;\Y)$ and is continuous.  This completes the proof.
\end{proof}

We conclude with  an extrapolation result which is particularly adapted to singular integral operators discussed in the next section. We consider the function space $\E_{w,loc} (\R_+;\X )$  defined by
\[
\E_{w,loc} (\R_+; \X ) := \left\{ u : \R_+ \to \X : u \chi_{[0,\tau]} \in \E_w (\R_+; \X ) \text{ for every } \tau >0 \right\}.
\]
This space is a Fr\'echet space for the natural topology. The notion of singular integral operators extends in a natural way to continuous, linear operators on the space  $L^p_{loc}(\R_+; X)$ ($p\in(1,\infty)$). We say that a kernel $K:\R_+\times \R_+\rightarrow \cL(X)$ supported on $\{(t,s)\in\R_+\times \R_+: s<t \}$ {\em satisfies the $(D_{r,+})$ (resp. $(D_{r,+}')$) condition locally}, if the function $K\chi_{\{(t,s)\in \R_+\times\R_+: s<t < \tau   \}}$ satisfies the $(D_{r,+})$ (resp. $(D_{r,+}'))$ condition for every $\tau>0$.

\begin{corollary} \label{cor.exsio}
Let $T$ be a singular integral operator from $L^p_{loc} (\R_+ ; \X )$ into $L^p_{loc} (\R_+ ;\Y)$ ($p\in (1,\infty )$), associated with a kernel $K$ satisfying the conditions $(D_{1,+})$ and $(D_{r,+}')$ locally for every $r\in[1,\infty)$. Then for every rearrangement invariant Banach function space  $\E$ with Boyd indices $p_{\E}$, $q_{\E} \in (1, \infty)$, and for every weight $w\in A^-_{p_\E}(\R_+)$, $T$ extrapolates to a continuous linear operator from $\E_{w,loc} (\R_+; \X)$ into $\E_{w,loc} (\R_+;\Y)$.
\end{corollary}

\begin{proof}
Note that the operator $T_\tau$ ($\tau>0)$ given by $T_\tau f := Tf \, \chi_{[0,\tau]}$ is a singular integral operator on $L^p(\R_+; \X)$ associated with the kernel $K\chi_{\{(t,s)\in \R_+\times \R_+: s<t< \tau \}}$. By assumption on the kernel $K$ and by Theorem \ref{Boyd th}, the operators $T_\tau$ ($\tau>0$) are bounded from $\E_w(\R_+;\X)$ into $\E_w (\R_+;\Y)$. Consequently, since $\E_{w,loc}(\R_+; \X)\cap L^p_{loc} (\R_+; \X)$ is dense in $\E_{w,loc}(\R_+; \X)$, we easily find the desired claim.
\end{proof}

\section{Application to $L_p$-maximal regularity}\label{sec.apply}

\subsection{First order problems}

Let $A$ be a closed linear operator on a Banach space $X$. Let $p\in (1,\infty)$.
We say that the first order Cauchy problem
\begin{equation} \label{first order}
\dot{u} + A u = f \text{ on } \R_+ , \quad u(0) = 0.
\end{equation} 
has {\em $\E_w$-maximal regularity} if for each 
$f\in \E_{w, loc}(\R_+; X)$ there exists a unique function $u\in  W^{1,1}_{loc}(\R_+,X)$ such that $\dot u$, $Au\in \E_{w, loc}(\R_+; X)$, and such that $u$ solves (\ref{first order}).

It is well known that if the above Cauchy problem has $L^p$-maximal regularity (that is, $\E = L^p$ and $w=1$), then $-A$ generates an analytic $C_0$-semigroup $(e^{-tA})_{t\geq 0}$. Moreover, if $-A$ generates a $C_0$-semigroup, then the unique mild solution of \eqref{first order} is given by Duhamel's formula $u(t) = \int_0^\infty e^{-(t-s)A} f(s) \; \ud s$. In particular, if the semigroup is analytic (or merely differentiable), then 
\begin{align*}
& Au(t) = \int_0^t Ae^{-(t-s)A} f(s)\; \ud s \\
& \text{for every } f\in L^\infty (\R_+ ; \X) \text{ with compact support, and every } t\notin \supp f .
\end{align*}
Hence, summing up, if the first order Cauchy problem \eqref{first order} has $L^p$-maximal regularity, then the operator
\[
 T : L^p_{loc} (\R_+ ;X) \to L^p_{loc} (\R_+ ;X) , \quad f \mapsto Tf:= Au ,
\]
which is continuous by the closed graph theorem, is a singular integral operator with translation-invariant kernel given by
\[
 K(t,s) := \begin{cases} 
            Ae^{-(t-s)A} & \text{if } t>s>0 , \\[2mm]
            0 & \text{else} .
           \end{cases}
\]
It is well known that the analyticity of the semigroup implies that $K$ satisfies the standard conditions locally. Hence, the corollary to our main extrapolation theorem, Corollary \ref{cor.exsio}, yields the following result.

\begin{theorem} \label{thm.first.order}
 Assume that the first order Cauchy problem \eqref{first order} has $L^p$-maximal regularity for some $p\in (1,\infty )$. Then it has $\E_w$-maximal regularity for every rearrangement-invariant Banach function space $\E$ with Boyd indices $p_\E$, $q_\E\in (1,\infty )$ and every weight $w\in A_p^- (\R_+ )$. 
\end{theorem}

Conditions on the operator $A$ ensuring that the first order Cauchy problem \eqref{first order} has $L^p$-maximal regularity have been studied extensively in the literature, and we only refer the reader to the survey articles by Arendt \cite{Ar04} and Kunstmann \& Weis \cite{KuWe04}. 

\begin{remark}
A well-known result going back to De Simon \cite{De64} in the case of Hilbert spaces and to Sobolevskii \cite{So64} in the general case says that $L^p$-maximal regularity is independent of $p$, that is, if the first order Cauchy problem \eqref{first order} has $L^p$-maximal regularity for some $p\in (1,\infty)$, then it has $L^p$-maximal regularity for all $p\in(1,\infty)$. The proof of this extrapolation result is based on the theory of singular integral operators with operator-valued kernels, too. Namely, it is based on the fact that Calder\'on-Zygmund operators satisfy endpoint estimates (a weak $(1,1)$-estimate and an $L^\infty$-$BMO$ estimate) and the Marcinkiewicz interpolation theorem; see Benedek, Calder\'on \& Panzone \cite{BeCaPa62} for the corresponding extrapolation theorem and also Cannarsa \& Vespri \cite{CaVe86} and Hieber \cite{Hi99}.  

Subsequently, Pr\"uss \& Simonett showed that $L^p$ maximal regularity for some $p\in (1,\infty )$ implies $L^p_w$-maximal regularity for every $p\in (1,\infty)$ and every power weight $w(t):=t^{\beta}$ with  $\beta\in (0, p-1)$ \cite[Theorem 2.4]{PrSi04}. The extension for all Muckenhoupt power weights, that is, $\beta \in (-1,p-1)$, was established later by Haak \& Kunstmann \cite[Theorem 1.13]{HaKu07}, while Auscher \& Axelsson obtained the result for $\beta \in (-\infty , p-1)$, assuming, however, that $\X$ is a Hilbert space and $p=2$ \cite[Theorem 1.3]{AuAx11a}; see also \cite{AKMP12} for related results. Note that for $\beta\leq -1$, the power weights $w_\beta(t) = |t|^{\beta}$ fall out of the class of Muckenhoupt $A_p$-weights and even out of the larger class of Sawyer $A_p^+(\R)$-weights. However, they do belong to $A_p^- (\R_+ )$ for every $\beta \in (-\infty ,p-1)$, that is, for $\beta$ in the range considered by Auscher \& Axelsson (see also Remark \ref{rem.weights}). The power weights play an important role from the point of view of initial value problems. In fact, for initial values in the classical real interpolation spaces between $X$ and the domain $D_A$ ($p\in (1,\infty )$, $\beta\in (-1,p-1)$), the regularity of solutions can now be characterized. Note also that the value  $\beta =-1$ plays a essential role in a new approach to non-smooth boundary value problems studied in \cite{AuAx11a}.  See also an extension of these results to boundedness of maximal regularity operator $\cM_A$ on \emph{weighted} tent spaces \cite{AKMP12}.

Recently, in \cite{ChFi14}, Chill \& Fiorenza have shown that $L^p$-maximal regularity for the first order problem actually implies $\E_w$-maximal regularity for every rearrangement invariant Banach function space with Boyd indices $p_\E$, $q_\E\in (1,\infty )$ and every Muckenhoupt weight $w\in A_{p_\E}$. The corresponding abstract extrapolation result of \cite{ChFi14} is formulated for singular integral operators on $\R^N$ where Muckenhoupt weights seem to be appropriate, while we now see that in the application to extrapolation of maximal regularity of first order Cauchy problems on the half-line (and in Theorem \ref{Boyd th} above), the larger $A_p^- (\R_+ )$ classes are better adapted.  
\end{remark}

As in Chill \& Fiorenza \cite{ChFi14}, one may also apply our abstract extrapolation result to the nonautonomous first order Cauchy problem 
\begin{equation} \label{eq.first.order.nonautonomous}
 \dot u + A(t)u = f \text{ on } \R_+, \quad u(0) = 0 .
\end{equation}
Maximal regularity for this problem is defined similarly as for the autonomous Cauchy problem. There exists in the literature a set of various logically independent conditions on the operators $A(t)$ which imply wellposedness of this nonautonomous Cauchy problem, and sometimes also $L^p$-maximal regularity. In the so-called parabolic case, we mention for wellposedness the Kato-Tanabe conditions \cite{KaTa62}, \cite[Section 5.3]{Ta79}, the Acquistapace-Terreni conditions \cite{AcTe87}, the conditions from \cite[Theorem 6.1, p.150]{Pa83}, or -- in Hilbert spaces -- the conditions from  \cite[Th\'eor\`eme 1, p.670]{DaLi87VIII}. In \cite[Section 7]{ChFi14} several results from the literature were collected showing that if $(A(t))_{t\in\R_+}$ is a family of closed, linear operators satisfying the Kato-Tanabe conditions or if both $(A(t))_{t\in\R_+}$ and $(A(t)')_{t\in\R_+}$ satisfy the Acquistapace-Terreni conditions, then $(A(t))_{t\in\R_+}$ generates an evolution family $(U(t,s))_{t\geq s\geq 0}$, the solution 
of \eqref{eq.first.order.nonautonomous} is given by 
\[
 u(t) = \int_0^t U(t,s)f(s)\; \ud s ,
\]
and the kernel
\[
 K(t,s) = \begin{cases}
           A(t)U(t,s) & \text{if } t>s>0 , \\[2mm]
           0 & \text{else} ,
          \end{cases}
\]
associated with the maximal regularity operator $Tf := A(\cdot )u$ is a standard kernel. From these observations and Corollary \ref{cor.exsio}, we immediately obtain the following theorem.

\begin{theorem} \label{thm.firstorder.nonautonomous}
Assume that the family $(A(t))_{t\in\R_+}$ satisfies the Kato-Tanabe conditions or that both $(A(t))_{t\in\R_+}$ and $(A(t)')_{t\in\R_+}$ satisfy the Acquistapace-Terreni conditions. Assume further that the nonautonomous first order Cauchy problem \eqref{eq.first.order.nonautonomous} has $L^p$-maximal regularity. Then it has $\E_w$-maximal regularity for every rearrangement invariant Banach function space $\E$ with Boyd indices $p_\E$, $q_\E\in (1,\infty )$ and every weight $w\in A_p^- (\R_+ )$.
\end{theorem}

By Hieber \& Monniaux \cite[Theorem 3.1, Theorem 3.2]{HiMo00}, the Acquistapace-Terreni conditions always imply $L^p$-maximal regularity for every $p\in (1,\infty )$ when the underlying space is a Hilbert space. In general UMD-spaces, the Acquistapace-Terreni conditions imply $L^p$-maximal regularity, if the sectoriality condition is replaced by the stronger {\em $R$-sectoriality}; we refer to \v{S}trkalj \cite[Satz 4.2.6]{St00diss} for this result, and to \cite{KuWe04} and the references therein for the concept of $R$-boundedness.  

An example of a nonautonomous parabolic equation which can be rewritten as a nonautonomous Cauchy problem of the form \eqref{eq.first.order.nonautonomous} in a Hilbert space and with a family $(A(t))$ of operators satisfying Acquistapace-Terreni conditions has been described by Yagi \cite[Theorem 4.1]{Ya90}. That example fits into a more general framework where the operators $A(t)$ come from sesquilinear forms having constant form domain and satisfying a H\"older continuity condition; see Ouhabaz \& Spina \cite[Theorem 3.3]{OuSp10}. Combining our Theorem \ref{thm.firstorder.nonautonomous} with \cite[Theorem 3.3]{OuSp10}, we obtain the following result.

\begin{corollary}
Let $H$ and $V$ be Hilbert spaces such that $V$ is densely and continuously embedded into $H$. Let $(a(t))_{t\geq 0}$ be a family of sesquilinear forms on $V$. Assume that 
\begin{itemize}
 \item[(a)] $|a(t,u,v)| \leq M \, \| u\|_V \, \| v\|_V$ for some $M\geq 0$ and all $t\geq 0$, $u$, $v\in V$, 
 \item[(b)] $|a(t,u,u)| + \omega \, \| u\|_H^2 \geq \eta \, \| u\|_V^2$ for some $\omega$, $\eta >0$ and all $t\geq 0$, $u\in V$, 
 \item[(c)] $|a(t,u,v) - a(s,u,v)|\leq K\, |t-s|^\beta \, \| u\|_V \, \| v\|_V$ for some $K\geq 0$, $\beta >\frac12$ and all $t$, $s\geq 0$, $u$, $v\in V$.
\end{itemize}
Let $A(t)$ be the operator on $H$ associated with the form $a(t)$, that is, 
\begin{align*}
 D(A(t)) & := \{ u\in V : \exists f\in H \, \forall v\in V : a(t,u,v) = \langle f,v\rangle_H \} , \\
 A(t)u & := f .
\end{align*}
Then the nonautonomous Cauchy problem \eqref{eq.first.order.nonautonomous} on $H$ has $\E_w$-maximal regularity for every rearrangement invariant Banach function space $\E$ with Boyd indices $p_\E$, $q_\E\in (1,\infty )$ and every weight $w\in A_p^- (\R_+ )$.
\end{corollary}

\subsection{Second order problems}

Let $A$ and $B$ be two closed linear operators on a Banach space $X$. Similarly as before, we say that the second order Cauchy problem
\begin{equation} \label{second order}
\ddot{u} + B\dot{u} + A u = f \text{ on } \R_+ , \quad u(0) =
\dot{u}(0) = 0.
\end{equation}
has $\E_w$-maximal regularity if for every $f\in \E_{w,loc} (\R_+ ;X)$ it admits a unique strong solution $u\in W^{2,1}_{loc} (\R_+ ;X)$ such that $u$, $\dot u$, $\ddot u$, $Au$, $B\dot u\in \E_{w,loc} (\R_+ ;X)$. It has been shown in  Chill \& Srivastava \cite[Proposition 2.2]{ChSr05} that if the second order problem has $L^p$-maximal regularity, then there exists a so-called {\em sine family} $S\in C(\R_+ ;\cL (X)) \cap C^\infty ((0,\infty );\cL (X,D_A \cap D_B))$ such that the unique strong solution of \eqref{second order} has the form $u=S*f$. As a consequence, the operators $f\mapsto Au$, $f\mapsto B\dot u$ and $f\mapsto \ddot u$ are singular integral operators on $L^p_{loc} (\R_+ ;X)$. In \cite[Theorem 4.2]{ChSr08}, this fact was exploited in order to show that $L^p$-maximal regularity is independent of $p\in (1,\infty )$ by showing that the convolution kernels $AS(\cdot )$, $B\dot S(\cdot )$ and $\ddot S(\cdot )$ satisfy the H\"ormander conditions. Later, in \cite[Theorem 6]{ChKr14}, the present 
authors improved this result by showing that the relevant kernels satisfy the standard conditions locally, and thus obtained an extrapolation result into weighted rearrangement invariant Banach function spaces using Muckenhoupt weights \cite[Theorem 1]{ChKr14}. The following extends this latter result.  

\begin{theorem} \label{thm.mrextra.second}
 Assume that the second order Cauchy problem \eqref{second order} has $L^p$-maximal regularity for some $p\in (1,\infty )$. Then it has $\E_w$-maximal regularity for every rearrangement-invariant Banach function space $\E$ with Boyd indices $p_\E$, $q_\E\in (1,\infty )$ and every weight $w\in A_p^- (\R_+ )$. 
\end{theorem}

An example of a second order Cauchy problems having $L^p$-maximal regularity can be found, for example, in Chill \& Srivastava \cite[Section 4]{ChSr05}. In that example, the underlying Banach space is a UMD-space with property $(\alpha )$, the operator $A$ admits an $RH^\infty$-functional calculus on a sector, and $B = \alpha A^\varepsilon$ for an appropriate choice of $\varepsilon\in [\frac12 ,1]$ and $\alpha >0$; see \cite{ChSr05} for the precise assumptions.  

\subsection{Volterra equations and fractional order problems}

Let $A$ be a closed linear operator on a Banach space $X$, and let $a\in L^1_{loc} (\R_+ )$. We consider the abstract Volterra equation
\begin{equation} \label{eq.Volterra}
u + A a*u = f \text{ on } \R_+ , \quad u(0) = 0.
\end{equation}
This Volterra equation is well-posed in the sense of \cite[Definition 1.2, p. 31]{Pr93} if and only if there exists a {\em resolvent family} $(S(t))_{t\in\R_+}$ which is, by definition, a strongly continuous family of bounded, linear operators on $X$ such that $S(0)=I$, $S(t)A\subseteq AS(t)$ for every $t\geq 0$ and
\[
 S(t)x = x + \int_0^t a(t-s) AS(s)x\; \ud s \text{ for every } t\geq 0, \, x\in D(A) ; 
\]
compare with \cite[Proposition 1.1, p. 32]{Pr93}. If \eqref{eq.Volterra} is well-posed, then the solution $u$ of \eqref{eq.Volterra} for continuous $f$ is given by
\[
 u(t) = \frac{d}{dt} \int_0^t S(t-s)f(s) \; \ud s ,
\]
or, if $f\in W^{1,1}_{loc} (\R_+ ;X)$, 
\[
 u(t) = S(t)f(0) + \int_0^t S(t-s)\dot f(s)\; \ud s ;
\]
\cite[Proposition 1.2, p. 33]{Pr93}. Assume now that \eqref{eq.Volterra} {\em is} well-posed. Given a rearrangement invariant Banach function space $\E$ and a weight $w$ on $\R_+$, we say that the Volterra equation \eqref{eq.Volterra} has $\E_w$-maximal regularity if for every right-hand side $f$ of the form $a*g$ with $g\in \E_{w,loc} (\R_+ ;X)$ the unique solution $u$ is of the form $u=a*v$ with $v\in\E_{w,loc} (\R_+ ;X)$. 

Assume that the kernel $a$ is of subexponential growth, that is, the Laplace integral 
\[
 \hat{a} (\lambda ) := \int_0^\infty e^{-\lambda t} a(t) \, \ud t 
\]
converges absolutely for every $\lambda \in\C$ with ${\rm Re}\, \lambda >0$. In other words, the abscissa of absolute convergence of the Laplace transform $\hat{a}$ is less or equal to $0$. Following \cite[Definition 3.1, p. 68]{Pr93} we call the Volterra equation \eqref{eq.Volterra} {\em parabolic} if $\hat{a} (\lambda )\not= 0$ and $I + \hat{a} (\lambda ) A$ is invertible whenever ${\rm Re}\, \lambda >0$, and if there exists a constant $C\geq 0$ such that 
\[
 |(I+\hat{a}(\lambda )A)^{-1}|_{\cL (X)} \leq C \text{ for every } \lambda\in\C , \, {\rm Re}\, \lambda >0 .
\]
And following \cite[Definition 3.3, p. 69]{Pr93} the kernel $a$ is called {\em $k$-regular} if there is a constant $C\geq 0$ such that
\[
 |\lambda^n \hat{a}^{(n)} (\lambda ) | \leq C\, |\hat{a}(\lambda )| \text{ for every } \lambda\in\C , \, {\rm Re}\, \lambda >0 , \, 0\leq n\leq k .
\]

\begin{theorem} \label{thm.volterra}
Assume that the kernel $a$ is $2$-regular. Assume further that the Volterra equation \eqref{eq.Volterra} is parabolic and that it has $L^p$-maximal regularity for some $p\in (1,\infty )$. Then it has $\E_w$-maximal regularity for every rearrangement invariant Banach function space $\E$ with Boyd indices $p_\E$, $q_\E\in (1,\infty )$ and every weight $w\in A_{p_\E}^- (\R_+ )$.  
\end{theorem}

Conditions implying $L^p$-maximal regularity for the Volterra equation can be found, for example, in Pr\"u{\ss} \cite[Theorem 8.7, p.227]{Pr93} or Zacher \cite[Theorem 3.1]{Za05}.

\begin{proof}
Assume that the Volterra equation \eqref{eq.Volterra} has $L^p$-maximal regularity for some $p\in (1,\infty )$. By the representation above, if $f=a*g$ for some continuous $g$, then the unique solution is of the form 
\[
u = \frac{d}{dt} (S*a*g) . 
\]
Since $\frac{d}{dt} (S*g)$ is the unique solution of the problem \eqref{eq.Volterra} with the right-hand side $f$ replaced by $g$, and since in particular $S*g$ is continuously differentiable, we have
\[
 u = a * [ \frac{d}{dt} (S * g) ] .
\]
Since the convolution by $a$ is an injective operator (this follows, in particular, from the assumption of parabolicity), the assumption of $L^p$-maximal regularity implies that the operator 
\begin{align*}
 T : C (\R_+ ;X) & \to C(\R_+ ;X) , \\
      g & \mapsto \frac{d}{dt} (S*g) 
\end{align*}
extends to a continuous, linear operator on $L^p (\R_+ ;X)$. Since the Volterra equation \eqref{eq.Volterra} is parabolic and since $a$ is $2$-regular, and by \cite[Theorem 3.1, p. 73]{Pr93}, the resolvent family $S$ is continuously differentiable from $(0,\infty )$ with values in $\cL (X)$. Hence, if $g\in L^\infty_c (\R_+ ;X)$ and $t\not\in\supp g$, then $Tg (t) = \int_0^t \dot S(t-s) g(s) \; ds$. In other words, $T$ is a singular integral operator of convolution type on $L^p_{loc} (\R_+ ;X)$ associated with the kernel
\[
 K(t,s) = \begin{cases} 
           \dot S (t-s) & \text{if } t>s , \\[2mm]
           0 & \text{else.}
          \end{cases}
\]
This kernel is obviously supported in $\{ (t,s)\in \R_+\times\R_+ : t>s\}$. Moreover, by \cite[Theorem 3.1, p.73]{Pr93} again, the kernel satisfies the standard conditions locally. The claim now follows from Corollary \ref{cor.exsio}.
\end{proof}

\begin{remark} \label{rem.parabolic}
For every $\theta\in (0,\pi )$ we define the sector 
\[
 \Sigma_\theta := \{ \lambda\in\C : |{\rm arg}\, \lambda| \leq \theta \}  
\]
Let $a\in L^1_{loc} (\R_+ )$ be of subexponential growth. We say that $a$ is {\em $\theta$-sectorial} if 
\[
 \hat{a} (\lambda ) \in \Sigma_\theta \text{ for all } \lambda\in\C , \, {\rm Re}\,\lambda >0 ,
\]
and we define its sectoriality angle
\[
 \theta_a := \inf \{ \theta\in (0,\pi ) : a \text{ is } \theta\text{-sectorial} \}.
\]
Moreover, we call a closed, linear operator $A$ on a Banach space {\em $\theta$-sectorial} if 
\begin{align*} 
 & \sigma (A) \subseteq \Sigma_\theta \cup \{ 0\} , \text{ and} \\
 & \sup_{\lambda\not\in\Sigma_{\theta'}} |\lambda (\lambda -A)^{-1} |_{\cL (X)} <\infty \text{ for every } \theta'\in (\theta ,\pi ) ,
\end{align*}
and we also define its sectoriality angle
\[
 \theta_A := \inf \{ \theta\in (0,\pi ) : A \text{ is } \theta\text{-sectorial} \}.
\]
We simply say that the kernel $a$ and the operator $A$ are {\em sectorial} if $\theta_a$, $\theta_A\in (0,\pi )$, that is, if they are $\theta$-sectorial for some $\theta\in (0,\pi )$.  Then it follows easily from the above definitions that the Volterra equation \eqref{eq.Volterra} is parabolic if $a$ and $A$ are sectorial and
\[
 \theta_a + \theta_A <\pi ;
\]
see also \cite[Proposition 3.1, p. 69]{Pr93}.
\end{remark}

A special case of the Volterra equation \eqref{eq.Volterra} is the fractional order Cauchy problem
\begin{equation} \label{eq.fractional}
 \frac{d^\alpha}{dt^\alpha} u + A u = f \text{ on } \R_+ , \quad u(0) = 0 ,
\end{equation}
for some $\alpha\in (0,2)$ and a closed, linear operator $A$ on some Banach space $X$. The fractional derivative is here defined by
\[
 \frac{d^\alpha}{dt^\alpha} u := \begin{cases}
                                 \frac{d}{dt} (k_{1-\alpha} * u)  & \text{if } \alpha\in (0,1) , \\[2mm]
                                 \frac{d^2}{dt^2} (k_{2-\alpha} * u)  & \text{if } \alpha\in [1,2) ,
                                 \end{cases}
\]
with 
\[
 k_\beta (t) := \frac{1}{\Gamma (\beta )} t^{\beta-1} \qquad (\beta >0 , \, t>0).
\]
Note that $k_\beta * k_\gamma = k_{\beta +\gamma}$ for every $\beta$, $\gamma>0$, that is, $(k_\beta )_{\beta >0}$ forms a convolution semigroup. One easily checks that convolution by $k_n$ corresponds to $n$-times integration for every natural number $n$, and hence it is appropriate to say that convolution by $k_\beta$ corresponds to $\beta$-times (fractional) integration, thus explaining the fractional derivative $\frac{d^\alpha}{dt^\alpha}$. The fractional order Cauchy problem \eqref{eq.fractional} can be transformed into the Volterra equation as can be seen by taking $a:= k_\alpha$ and by convolving both sides of equation \eqref{eq.fractional} by $a$. We say that the fractional order Cauchy problem \eqref{eq.fractional} is well-posed if the resulting Volterra equation is well-posed, and we say that it has $\E_w$-maximal regularity if for every $f\in \E_{w,loc} (\R_+ ;X)$ the unique mild solution $u = \frac{d}{dt} (S*k_\alpha*f)$ of \eqref{eq.fractional} satisfies $\frac{d^\alpha}{dt^\alpha} u$, $Au\in \E_
{w,loc}
 (\R_+ ;X)$. Comparing the respective definitions of maximal regularity one easily sees that the fractional order Cauchy problem \eqref{eq.fractional} has $\E_w$-maximal regularity if and only if the associated Volterra equation \eqref{eq.Volterra} has $\E_w$-maximal regularity. From this observation one obtains the following corollary.

\begin{corollary}
 Assume that $\alpha\in (0,2)$ and that $A$ is sectorial with $\theta_A\in (0,\frac{(2-\alpha)\pi}{2})$. Assume further that the fractional order Cauchy problem \eqref{eq.fractional} has $L^p$-maximal regularity for some $p\in (1,\infty )$. Then it has $\E_w$-maximal regularity for every rearrangement invariant Banach function space $\E$ with Boyd indices $p_\E$, $q_\E\in (1,\infty )$ and every weight $w\in A_{p_\E}^- (\R_+ )$.
\end{corollary}

\begin{proof}
Since $\hat{k}_\alpha (\lambda ) = \lambda^{-\alpha}$, one easily sees that $k_\alpha$ is sectorial with $\theta_{k_\alpha} = \frac{\alpha\pi}{2}$. Hence, by Remark \ref{rem.parabolic}, the associated Volterra equation \eqref{eq.Volterra} (with $a = k_\alpha$) is parabolic. Moreover, $k_\alpha$ is $2$-regular. The statement thus follows from Theorem \ref{thm.volterra}. 
\end{proof}

\begin{remark}
 Note that for $\alpha = 1$ we recover Theorem \ref{thm.first.order}. 
\end{remark}

\nocite{AuMoPo12}
\nocite{BeLo76}

\providecommand{\bysame}{\leavevmode\hbox to3em{\hrulefill}\thinspace}

\bibliographystyle{amsplain}

\def\cprime{$'$} 
\def\ocirc#1{\ifmmode\setbox0=\hbox{$#1$}\dimen0=\ht0 \advance\dimen0
  by1pt\rlap{\hbox to\wd0{\hss\raise\dimen0
  \hbox{\hskip.2em$\scriptscriptstyle\circ$}\hss}}#1\else {\accent"17 #1}\fi}

\providecommand{\bysame}{\leavevmode\hbox to3em{\hrulefill}\thinspace}
\providecommand{\MR}{\relax\ifhmode\unskip\space\fi MR }
\providecommand{\MRhref}[2]{%
  \href{http://www.ams.org/mathscinet-getitem?mr=#1}{#2}
}
\providecommand{\href}[2]{#2}

\end{document}